\documentclass[a4j,10pt,showkeys]{article}

\usepackage{amsmath,amsthm,amssymb}
\usepackage{geometry}
\usepackage{hyperref}
\usepackage[capitalize]{cleveref}
\usepackage{cancel}
\usepackage{mathtools}
\usepackage{bm}
\usepackage{color}
\usepackage{mathrsfs}

\newtheorem{theo}{Theorem}[section]
\newtheorem{lemm}[theo]{Lemma}
\newtheorem{prop}[theo]{Proposition}

\theoremstyle{definition}

\newtheorem{rem}[theo]{Remark}
\newtheorem{assum}{Assumption}

\newcommand{\bP}{\mathbb{P}}
\newcommand{\bE}{\mathbb{E}}
\newcommand{\bF}{\mathbb{F}}
\newcommand{\bR}{\mathbb{R}}
\newcommand{\bN}{\mathbb{N}}
\newcommand{\bX}{\mathbb{X}}
\newcommand{\bY}{\mathbb{Y}}
\newcommand{\bZ}{\mathbb{Z}}
\newcommand{\cF}{\mathcal{F}}
\newcommand{\cT}{\mathcal{T}}
\newcommand{\cB}{\mathcal{B}}
\newcommand{\cM}{\mathcal{M}}
\newcommand{\cU}{\mathcal{U}}
\newcommand{\cX}{\mathcal{X}}
\newcommand{\cY}{\mathcal{Y}}
\newcommand{\cZ}{\mathcal{Z}}

\newcommand{\sL}{\mathscr{L}}
\newcommand{\rd}{\,\mathrm{d}}

\newcommand{\relmiddle}[1]{\mathrel{}\middle#1\mathrel{}}
\newcommand{\1}{\mbox{\rm{1}}\hspace{-0.25em}\mbox{\rm{l}}}
\newcommand{\comp}{\mathrm{c}}
\newcommand{\CD}{^\mathrm{C}\!D^\alpha_{0+}}
\newcommand{\res}{|\hspace{-0.2mm}|\hspace{-0.2mm}|}

\providecommand{\keywords}[1]{\textbf{Keywords:} #1}

\makeatletter

\@addtoreset{equation}{section}
\makeatother
\def\widebar{\accentset{{\cc@style\underline{\mskip10mu}}}}
\numberwithin{equation}{section}
\def\rnum#1{\expandafter{\romannumeral #1}} 
\def\Rnum#1{\uppercase\expandafter{\romannumeral #1}}

\allowdisplaybreaks

\title{On the maximum principle for optimal control problems of stochastic Volterra integral equations with delay\footnote{This paper is accepted for publication in {\it Applied Mathematics and Optimization} following peer review.}}
\author{
	Yushi Hamaguchi\thanks{
	Graduate School of Engineering Science, Department of Systems Innovation,
	Osaka University.
	1-3, Machikaneyama, Toyonaka, Osaka, Japan.
	\href{mailto:hmgch2950@gmail.com}{hmgch2950@gmail.com}}
}

\begin{document}
\maketitle

\begin{abstract}
In this paper, we prove both necessary and sufficient maximum principles for infinite horizon discounted control problems of stochastic Volterra integral equations with finite delay and a convex control domain. The corresponding adjoint equation is a novel class of infinite horizon anticipated backward stochastic Volterra integral equations. Our results can be applied to discounted control problems of stochastic delay differential equations and fractional stochastic delay differential equations. As an example, we consider a stochastic linear-quadratic regulator problem for a delayed fractional system. Based on the maximum principle, we prove the existence and uniqueness of the optimal control for this concrete example and obtain a new type of explicit Gaussian state-feedback representation formula for the optimal control.
\end{abstract}
\keywords
Maximum principles;
stochastic delay Volterra integral equations;
anticipated backward stochastic Volterra integral equations;
fractional stochastic delay differential equations;
Gaussian state-feedback representation formula.

%93E20 Optimal stochastic control
%49K45 Optimality conditions for problems involving randomness
%60H20 Stochastic integral equations
%26A33 Fractional derivatives and integrals

%45G05 Singular nonlinear integral equations
%49N15 Duality theory (optimization)
%45D05 Volterra integral equations
%45B05 Fredholm integral equations
%34A08 Fractional ordinary differential equations and fractional differential inclusions
%93B52 Feedback control

%%%%%%%%%%%%%%%%
%%% Section
%%%%%%%%%%%%%%%%

\section{Introduction}

In this paper, we are interested in the Pontryagin's maximum principle for a general class of infinite horizon optimal control problems with the state dynamics given by
\begin{equation}\label{1: state}
	\begin{dcases}
	X^u(t)=\varphi(t)+\int^t_0b(t,s,X^u(s),X^u(s-\delta),u(s))\rd s+\int^t_0\sigma(t,s,X^u(s),X^u(s-\delta),u(s))\rd W(s),\ t\geq0,\\
	X^u(t)=\varphi(t),\ t\in[-\delta,0],
	\end{dcases}
\end{equation}
where $W(\cdot)$ is a multi-dimensional Brownian motion, $\delta\geq0$ is a given constant, $b$ and $\sigma$ are given deterministic maps, $\varphi(\cdot)$ is a given adapted process, and $u(\cdot)$ is a control process which takes vlaues in a convex subset of a Euclidean space. The controlled equation \eqref{1: state} is a stochastic Volterra integral equation (SVIE, for short) which has a ``finite delay'' of the form $X^u(s-\delta)$, and thus we call it a \emph{stochastic delay Volterra integral equation} (SDVIE, for short). Our objective is to find a control process which minimizes the discounted cost functional
\begin{equation}\label{1: cost}
	J_\lambda(u(\cdot))=\bE\Bigl[\int^\infty_0e^{-\lambda t}h(t,X^u(t),X^u(t-\delta),u(t))\rd t\Bigr],
\end{equation}
where $h$ is a real-valued deterministic function, and $\lambda\in\bR$ is a given discount rate.

If the coefficients $b(t,s,x_1,x_2,u)$ and $\sigma(t,s,x_1,x_2,u)$ do not depend on the time-parameter $t$, and if the free term is of the form $\varphi(t)=\varphi(t\wedge0)$, then SDVIE~\eqref{1: state} is reduced to a stochastic delay differential equation (SDDE, for short)
\begin{equation}\label{1: SDDE_state}
	\begin{dcases}
	\mathrm{d}X^u(t)=b(t,X^u(t),X^u(t-\delta),u(t))\rd t+\sigma(t,X^u(t),X^u(t-\delta),u(t))\rd W(t),\ t\geq0,\\
	X^u(t)=\varphi(t),\ t\in[-\delta,0].
	\end{dcases}
\end{equation}
If furthermore $\delta=0$, then the above SDDE becomes the well-known stochastic differential equation (SDE, for short) without delay. More importantly, SDVIE~\eqref{1: state} includes a class of fractional-order SDDEs of the form
\begin{equation}\label{1: FSDDE_state}
	\begin{dcases}
	\CD X^u(t)=b(t,X^u(t),X^u(t-\delta),u(t))+\sigma(t,X^u(t),X^u(t-\delta),u(t))\frac{\mathrm{d}W(t)}{\mathrm{d}t},\ t\geq0,\\
	X^u(t)=\varphi(t),\ t\in[-\delta,0],
	\end{dcases}
\end{equation}
where $\CD$ denotes the Caputo fractional derivative of order $\alpha\in(\frac{1}{2},1)$ defined by
\begin{equation*}
	\CD f(t):=\frac{1}{\Gamma(1-\alpha)}\frac{\mathrm{d}}{\mathrm{d}t}\int^t_0(t-s)^{-\alpha}\{f(s)-f(0)\}\rd s,\ t\geq0,
\end{equation*}
for suitable function $f:[0,\infty)\to\bR$. Here and elsewhere, $\Gamma(\alpha)=\int^\infty_0e^{-\tau}\tau^{\alpha-1}\rd\tau$ denotes the Gamma function. Indeed, by the definition (see \cite{ZhAgLiPeYoZh17,MoZhLoTeMo20}), an adapted process $X^u(\cdot)$ is called a solution to fractional SDDE \eqref{1: FSDDE_state} if it solves the equation
\begin{equation*}
	\begin{dcases}
	X^u(t)=\varphi(0)+\frac{1}{\Gamma(\alpha)}\int^t_0(t-s)^{\alpha-1}b(s,X^u(s),X^u(s-\delta),u(s))\rd s\\
	\hspace{2cm}+\frac{1}{\Gamma(\alpha)}\int^t_0(t-s)^{\alpha-1}\sigma(s,X^u(s),X^u(s-\delta),u(s))\rd W(s),\ t\geq0,\\
	X^u(t)=\varphi(t),\ t\in[-\delta,0].
	\end{dcases}
\end{equation*}
This equation can be seen as a controlled SDVIE with singular kernels in the sense that $\lim_{s\to t}(t-s)^{\alpha-1}=\infty$. Fractional differential systems are suitable tools to describe the dynamics of systems with memory effects and hereditary properties. There are many applications of fractional calculus in a variety of research fields including mathematical finance, physics, chemistry, biology, and other applied sciences. For detailed accounts of theory and applications of fractional calculus, see for example \cite{DaBa10,Di07,Ra10,SaKiMa87} and the references cited therein. In addition to the fractional derivative, if the state dynamics and/or the cost functional contain delay arguments, we face with an optimal control problem for a delayed fractional system, which can be used to model more realistic controlled dynamic systems which have memory, as well as delay hypotheses. In the deterministic case (that is, the case of $\sigma=0$ in \eqref{1: FSDDE_state}), Jajarmi and Baleanu~\cite{JaBa18} studied optimal control problems for a class of fractional-order dynamic systems with finite delay, which can be seen as a special case of our study. Recently, Zhang et al.~\cite{ZhAgLiPeYoZh17} and Moghaddam et al.~\cite{MoZhLoTeMo20} studied the fractional SDDE~\eqref{1: FSDDE_state} (without control) and proved the existence and uniqueness of the solution. However, to the best of our knowledge, there have been no studies on optimal control problems for fractional SDDEs. The analysis of such a stochastic control problem is therefore a natural and important topic in views of both theory and applications, and this is the main motivation of this paper.

It is well-known that the maximum principle is an important approach in solving optimal control problems (see the textbook~\cite{YoZh99} and references cited therein). In 1956, Boltyanski, Gamkrelidze, and Pontryagin~\cite{BoGaPo56} proposed the Pontryagin's maximum principle for the first time for deterministic control systems. Bismut~\cite{Bi78} and Peng~\cite{Pe90} studied stochastic control problems of SDEs and derived the maximum principle. Since then, many researchers have tried to generalize and refine the maximum principle in various classes of deterministic and stochastic control problems. Here, let us briefly review the related works. Bergounioux and Bourdin~\cite{BeBo20} investigated the necessary maximum principle for finite horizon deterministic Caputo fractional optimal control problems with terminal constraints. Lin and Yong~\cite{LiYo20} proved the necessary maximum principle for finite horizon control problems of deterministic singular Volterra equations and applied it to deterministic fractional differential systems. Yong~\cite{Yo06,Yo08} and Wang~\cite{WaT20} obtained the necessary maximum principle for finite horizon control problems of SVIEs. Chen and Wu~\cite{ChWu10} and {\O}ksendal, Sulem, and Zhang~\cite{OkSuZh11} showed the necessary maximum principle for finite horizon control problems of SDDEs. In the infinite horizon setting, Maslowski and Veverka~\cite{MaVe14} and Orrieri and Veverka~\cite{OrVe17} established necessary and sufficient maximum principles for SDEs with dissipative coefficients, respectively. Lastly, in our previous work~\cite{Ha21+}, we proved both necessary and sufficient maximum principles for infinite horizon discounted control problems of SVIEs.

The most important step in deriving maximum principles is the introduction of the \emph{adjoint equation} which corresponds to the variation of the state equation by means of a duality relation. Loosely speaking, in the finite horizon control problems of SDEs~\cite{Bi78,Pe90}, SDDEs~\cite{ChWu10,OkSuZh11}, and SVIEs~\cite{Yo06,Yo08,WaT20}, the corresponding adjoint equations become backward stochastic differential equations (BSDEs, for short), anticipated BSDEs (ABSDEs, for short), and backward stochastic Volterra integral equations (BSVIEs, for short), respectively. Alternatively, in the infinite horizon setting, the adjoint equations corresponding to SDEs~\cite{MaVe14,OrVe17} and SVIEs~\cite{Ha21+} turned out to be infinite horizon BSDEs and infinite horizon BSVIEs, respectively. The structure of the adjoint equation heavily relies on the underlying state dynamics. Therefore, in the literature, the methods deriving maximum principles have depended on the problems.

On the other hand, in our previous work~\cite{Ha21+}, we showed the applicability of the maximum principle for SVIEs to discounted control problems of several types of state dynamics such as SDEs, fractional SDEs, and stochastic integro-differential equations (see Section~4.3 in \cite{Ha21+}). We remark that a stochastic integro-differential equation can be seen as a dynamics with unbounded delay, and the delay terms are given by Lebesgue integrals of the past trajectories of states and controls. The corresponding adjoint equation turned out to be an infinite horizon ABSDE of Ito--Volterra type (or an infinite horizon backward stochastic integro-differential equation).

This paper is a continuation of our previous work \cite{Ha21+}. In this paper, we prove both necessary and sufficient maximum principles for infinite horizon discounted control problems of SDVIEs (with finite delay) and apply them to SDDEs and fractional SDDEs. To the best of our knowledge, maximum principles for fractional SDDEs and SDVIEs have not been obtained in the literature. Furthermore, our results on SDDEs generalize the frameworks of \cite{ChWu10,OkSuZh11} to the infinite horizon setting. It is remarkable that the adjoint equation corresponding to SDVIEs becomes a novel class of \emph{anticipated BSVIEs} (ABSVIEs, for short), which is beyond the class of ABSVIEs studied by Wen and Shi~\cite{WeSh20}. Surprisingly, all the above results can be obtained by using our previous results \cite{Ha21+} on SVIEs without delay. A simple but important idea is to ``lift up'' the dimension of the SDVIE so that the auxiliary state equation becomes a classical SVIE (without delay). This is a new feature of SVIEs which reveals an interesting difference from SDEs, since SDDEs cannot be changed into (finite dimensional) SDEs without delay. For more details, see \cref{sec: SDVIE}. A similar idea can be seen in Example 4.12 in \cite{Ha21+}. Therefore, combining our previous work~\cite{Ha21+} and the current paper, we see that the general results in \cite{Ha21+} provide a \emph{unified approach} to maximum principles for many kinds of state dynamics such as SDEs, SDDEs, fractional SDEs, fractional SDDEs, stochastic integro-differential equations, SVIEs, and SDVIEs. The corresponding adjoint equations can be obtained by appropriate transformations of BSVIEs. An interesting fact is that the maximum principle in \cite{Ha21+} was proved without It\^{o}'s formula or BSDE theory. This means that we can recover some known maximum principles for SDEs, SDDEs, and so on, without It\^{o}'s formula or BSDE theory.

Furthermore, in this paper, as an example of our general theory, we investigate an infinite horizon linear-quadratic (LQ, for short) regulator problem for a fractional SDDE with constant coefficients. This problem can be seen as a generalization of the (finite horizon) deterministic LQ delay fractional optimal control problem studied by Jajarmi and Baleanu~\cite{JaBa18} to the stochastic case. Based on the maximum principle, we show that there exists a unique optimal control for this problem. Moreover, we obtain an explicit state-feedback representation formula for the optimal control process, which is given by the following form:
\begin{equation*}
	(\text{optimal control})=(\text{constant})\times(\text{delayed optimal state})+(\text{Gaussian process}).
\end{equation*}
Here, the Gaussian process is a stochastic convolution of a deterministic function with respect to the Brownian motion, and the function is expressed by using a Fredholm resolvent of a kernel which is determined by the model parameters only. For more detailed assertions, see \cref{LQ: theo_characterization}. In this paper, we call the above expression a \emph{Gaussian state-feedback representation formula} for the optimal control. This kind of representation formula appears for the first time in the literature. For relevant studies on feedback representations of optimal controls, we refer to \cite{BoCoMa12,CoMa11,Ma19,AbMiPh19}. Bonaccorsi, Confortola, and Mastrogiacomo~\cite{BoCoMa12} and Confortola and Mastrogiacomo~\cite{CoMa11} studied finite horizon stochastic control problems for SVIEs with completely monotone kernels. They reformulated the state equations into infinite-dimensional correspondences and obtained the optimal feedback law by means of Hamilton--Jacobi--Bellman equations, together with the semigroup methods. These methods were extended to the infinite horizon setting by Mastrogiacomo~\cite{Ma19}. Also, Abi Jaber, Miller, and Pham~\cite{AbMiPh19} studied finite horizon LQ control problems for SVIEs with completely monotone kernels, and obtained a feedback formula for the optimal control by means of infinite-dimensional integral operator Riccati equations (see also \cite{AbMiPh21}). Compared to the aforementioned papers, we do not use infinite-dimensional calculus. In place, based on the maximum principle, we derive \emph{linear Fredholm integral equations} which characterize the optimal control law. By using the corresponding Fredholm resolvents, we can solve these equations and obtain the Gaussian state-feedback representation formula for the optimal control. Therefore, our method deriving the feedback formula is completely different from the aforementioned papers. We emphasize that, although our result relies on the special structure of the LQ regulator problem, the Gaussian state-feedback representation formula holds for the fractional dynamics with delay, which was not discussed in the literature.

We remark that the control domains considered in this paper and our previous work~\cite{Ha21+} are assumed to be convex. This assumption makes it possible to apply the convex variation to derive the maximum principle. On the other hand, if the control domain is non-convex, the convex variation method is no longer available, and we have to consider the spike variation instead. Recently, Wang~\cite{WaT20} and Wang and Yong~\cite{WaTYo22} studied the spike variation method for controlled SVIEs with non-convex control domains. Unfortunately, due to the regularity assumptions on the kernels, their results cannot be applied to SVIEs or SDVIEs with singular kernels such as fractional SDEs or fractional SDDEs. Thus, the theory of spike variations for SVIEs and SDVIEs with singular kernels and non-convex control domains is still not fully developed. We hope to report the relevant results in the near future.

The rest of this paper is organized as follows. In \cref{sec: SVIE}, we briefly recall the results of our earlier work \cite{Ha21+} on the maximum principles for SVIEs without delay. In \cref{sec: SDVIE}, we prove necessary and sufficient maximum principles for general SDVIEs. The main result is \cref{SDVIE: theo_MP}. In \cref{sec: SDDE}, we investigate discounted control problems for SDDEs and explain an idea to obtain the appropriate adjoint equation from an infinite horizon ABSVIE. In \cref{sec: FSDDE}, we apply our general results to discounted control problems of fractional SDDEs. In Subsection~\ref{subsec: FSDDE_LQ}, we investigate an infinite horizon LQ regulator problem for a fractional SDDE. The main result in this subsection is \cref{LQ: theo_characterization}.

%%%%%%%%%%%%%%

\subsection*{Notation}

%%%%%%%%%%%%%%

For each $d_1,d_2\in\bN$, we denote the space of $(d_1\times d_2)$-matrices by $\bR^{d_1\times d_2}$, which is endowed with the Frobenius norm denoted by $|\cdot|$. We define $\bR^{d_1}:=\bR^{d_1 \times 1}$, that is, each element of $\bR^{d_1}$ is understood as a column vector. We denote by $\langle\cdot,\cdot\rangle$ the usual inner product in a Euclidean space. For each  matrix $A$, $A^\top$ denotes the transpose of $A$. For each scalar-valued differentiable function $f$ on $\bR^{d_1}$ with $d_1\in\bN$, the derivative $\partial_xf(x')\in\bR^{d_1}$ of $f$ at $x'\in\bR^{d_1}$ is understood as a column vector. If $f$ is $\bR^{d_2}$-valued with $d_2\in\bN$, the derivative $\partial_xf(x')$ is understood as a $(d_2\times d_1)$-matrix. 

$(\Omega,\cF,\bP)$ is a complete probability space, and $W(\cdot)=(W_1(\cdot),\dots,W_d(\cdot))^\top$ is a $d$-dimensional Brownian motion on $(\Omega,\cF,\bP)$ with $d\in\bN$. $\bF=(\cF_t)_{t\geq0}$ denotes the filtration generated by $W(\cdot)$ and augmented by $\bP$. We define $\cF_\infty:=\bigvee_{t\geq0}\cF_t$. For each $t\geq0$, $\bE_t[\cdot]:=\bE[\cdot|\cF_t]$ is the conditional expectation given by $\cF_t$.

We define
\begin{equation*}
	\Delta[0,\infty):=\{(t,s)\in[0,\infty)^2\,|\,0\leq t\leq s<\infty\}\ \text{and}\ \Delta^\comp[0,\infty):=\{(t,s)\in[0,\infty)^2\,|\,0\leq s\leq t<\infty\}.
\end{equation*}

Let $U$ be a nonempty Borel subset of a Euclidean space, and let $p\in[1,\infty)$ and $\beta\in\bR$ be fixed. $L^p_{\cF_\infty}(\Omega;U)$ denotes the set of $U$-valued and $\cF_\infty$-measurable $L^p$-random variables. Define
\begin{equation*}
	L^{p,\beta}(0,\infty;U):=\Bigl\{f:[0,\infty)\to U\,|\,\text{$f$ is measurable},\ \int^\infty_0e^{p\beta t}|f(t)|^p\rd t<\infty\Bigr\}.
\end{equation*}
It is easy to see that $L^{p',\beta'}(0,\infty;U)\subset L^{p,\beta}(0,\infty;U)$ if $p\leq p'$ and $\beta<\beta'$. However, $p<p'$ does not imply $L^{p',\beta}(0,\infty;U)\subset L^{p,\beta}(0,\infty;U)$. We define
\begin{equation*}
	L^p(0,\infty;U):=L^{p,0}(0,\infty;U)\ \text{and}\ L^{p,*}(0,\infty;U):=\bigcup_{\beta\in\bR}L^{p,\beta}(0,\infty;U).
\end{equation*}
Also, we define
\begin{equation*}
	L^{p,\beta}_{\cF_\infty}(0,\infty;U):=\Bigl\{f:\Omega\times[0,\infty)\to U\,|\,\text{$f$ is $\cF_\infty\otimes\cB([0,\infty))$-measurable},\ \bE\Bigl[\int^\infty_0e^{p\beta t}|f(t)|^p\rd t\Bigr]<\infty\Bigr\},
\end{equation*}
\begin{equation*}
	L^{p,\beta}_\bF(0,\infty;U):=\{f(\cdot)\in L^{p,\beta}_{\cF_\infty}(0,\infty;U)\,|\,f(\cdot)\ \text{is adapted}\},
\end{equation*}
and
\begin{equation*}
	\sL^{p,\beta}_\bF(0,\infty;U):=\left\{f:\Omega\times[0,\infty)^2\to U\relmiddle|\begin{aligned}&\text{$f$ is $\cF\otimes\cB([0,\infty)^2)$-measurable},\\&\text{$f(t,\cdot)$ is adapted for a.e.\ $t\geq0$},\\&\bE\Bigl[\int^\infty_0e^{p\beta t}\int^\infty_0|f(t,s)|^p\rd s\rd t\Bigr]<\infty\end{aligned}\right\}.
\end{equation*}
If $U$ is a Euclidean space, these spaces are Banach spaces with the norms, for example,
\begin{equation*}
	L^{p,\beta}_{\cF_\infty}(0,\infty;U)\ni f(\cdot)\mapsto \bE\Bigl[\int^\infty_0e^{p\beta t}|f(t)|^p\rd t\Bigr]^{1/p}
\end{equation*}
and
\begin{equation*}
	\sL^{p,\beta}_\bF(0,\infty;U)\ni f(\cdot,\cdot)\mapsto\bE\Bigl[\int^\infty_0e^{p\beta t}\int^\infty_0|f(t,s)|^p\rd s\rd t\Bigr]^{1/p}.
\end{equation*}
$L^p_{\cF_\infty}(0,\infty;U)$, $L^p_\bF(0,\infty;U)$, $\sL^p_\bF(0,\infty;U)$, $L^{p,*}_{\cF_\infty}(0,\infty;U)$, $L^{p,*}_\bF(0,\infty;U)$, and $\sL^{p,*}_\bF(0,\infty;U)$ are defined by similar manners as before.

For each $z(\cdot)\in L^2_\bF(0,\infty;\bR^{d_1\times d})$ with $d_1\in\bN$, the stochastic integral of $z(\cdot)$ with respect to the $d$-dimensional Brownian motion $W(\cdot)$ is defined by
\begin{equation*}
	\int^t_0z(s)\rd W(s):=\sum^d_{k=1}\int^t_0z^k(s)\rd W_k(s),\ t\geq0,
\end{equation*}
which is a $d_1$-dimensional square-integrable martingale. Here and elsewhere, for each $z\in\bR^{d_1\times d}$, $z^k\in\bR^{d_1}$ denotes the $k$-th column vector for $k=1,\dots,d$.

Denote by $\cM^{2,\beta}_\bF(0,\infty;\bR^{d_1}\times\bR^{d_1\times d})$ with $d_1\in\bN$ and $\beta\in\bR$ the set of pairs $(y(\cdot),z(\cdot,\cdot))\in L^{2,\beta}_\bF(0,\infty;\bR^{d_1})\times\sL^{2,\beta}_\bF(0,\infty;\bR^{d_1\times d})$ such that
\begin{equation*}
	y(t)=\bE[y(t)]+\int^t_0z(t,s)\rd W(s)
\end{equation*}
for a.e.\ $t\geq0$, a.s. Clearly, the space $\cM^{2,\beta}_\bF(0,\infty;\bR^{d_1}\times\bR^{d_1\times d})$ is a closed subspace of the product space $L^{2,\beta}_\bF(0,\infty;\bR^{d_1})\times\sL^{2,\beta}_\bF(0,\infty;\bR^{d_1\times d})$, and thus it is a Hilbert space. As before, we define $\cM^2_\bF(0,\infty;\bR^{d_1}\times\bR^{d_1\times d}):=\cM^{2,0}_\bF(0,\infty;\bR^{d_1}\times\bR^{d_1\times d})$ and $\cM^{2,*}_\bF(0,\infty;\bR^{d_1}\times\bR^{d_1\times d}):=\bigcup_{\beta\in\bR}\cM^{2,\beta}_\bF(0,\infty;\bR^{d_1}\times\bR^{d_1\times d})$.

Lastly, for each $f\in L^{p,*}(0,\infty;\bR_+)$, define $[f]_p:\bR\to[0,\infty]$ by
\begin{equation*}
	[f]_p(\rho):=\Bigl(\int^\infty_0e^{-p\rho t}f(t)^p\rd t\Bigr)^{1/p},\ \rho\in\bR.
\end{equation*}
Note that the function $[f]_p$ is non-increasing and satisfies $\lim_{\rho\to\infty}[f]_p(\rho)=0$.

%%%%%%%%%%%%%%%%
%%% Section
%%%%%%%%%%%%%%%%

\section{Discounted control problems for SVIEs}\label{sec: SVIE}

In this section, we briefly recall our earlier work~\cite{Ha21+}. For more detailed discussions, see Section~4 in \cite{Ha21+}.

For each $\mu\in\bR$, define the set of \emph{control processes} by $\cU_{-\mu}:=L^{2,-\mu}_\bF(0,\infty;U)$, where $U$ is a convex body (i.e., $U$ is convex and has a nonempty interior) in $\bR^\ell$ with $\ell\in\bN$. For each control process $u(\cdot)\in\cU_{-\mu}$, define the corresponding \emph{state process} $X^u(\cdot)$ as the solution of the following controlled SVIE:
\begin{equation*}
	X^u(t)=\varphi(t)+\int^t_0b(t,s,X^u(s),u(s))\rd s+\int^t_0\sigma(t,s,X^u(s),u(s))\rd W(s),\ t\geq0,
\end{equation*}
where $\varphi(\cdot)$ is a given process, and $b,\sigma$ are given deterministic maps. In order to measure the performance of $u(\cdot)$ and $X^u(\cdot)$, we consider the following \emph{discounted cost functional}:
\begin{equation*}
	J_\lambda(u(\cdot)):=\bE\Bigl[\int^\infty_0e^{-\lambda t}h(t,X^u(t),u(t))\rd t\Bigr],
\end{equation*}
where $h$ is a given $\bR$-valued deterministic function and $\lambda\in\bR$ is a \emph{discount rate}. The stochastic control problem is a problem to seek a control process $\hat{u}(\cdot)\in\cU_{-\mu}$ such that
\begin{equation*}
	J_\lambda(\hat{u}(\cdot))=\inf_{u(\cdot)\in\cU_{-\mu}}J_\lambda(u(\cdot)).
\end{equation*}
If it is the case, we call $\hat{u}(\cdot)$ an \emph{optimal control}.

%% Assumption

\begin{assum}\label{SVIE: assum}
\begin{itemize}
\item[(i)]
$\varphi(\cdot)\in L^{2,-\mu}_\bF(0,\infty;\bR^n)$.
\item[(ii)]
$b:\Delta^\comp[0,\infty)\times\bR^n\times\bR^\ell\to\bR^n$ and $\sigma:\Delta^\comp[0,\infty)\times\bR^n\times\bR^\ell\to\bR^{n\times d}$ are measurable; $b(t,s,x,u)$ and $\sigma(t,s,x,u)$ are continuously differentiable in $(x,u)\in\bR^n\times\bR^\ell$ for a.e.\ $(t,s)\in\Delta^\comp[0,\infty)$; there exist $K_{b,x},K_{b,u}\in L^{1,*}(0,\infty;\bR_+)$ and $K_{\sigma,x},K_{\sigma,u}\in L^{2,*}(0,\infty;\bR_+)$ such that
\begin{align*}
	&|b(t,s,x,u)-b(t,s,x',u')|\leq K_{b,x}(t-s)|x-x'|+K_{b,u}(t-s)|u-u'|,\\
	&|\sigma(t,s,x,u)-\sigma(t,s,x',u')|\leq K_{\sigma,x}(t-s)|x-x'|+K_{\sigma,u}(t-s)|u-u'|,
\end{align*}
for any $x,x'\in\bR^n$ and $u,u'\in\bR^\ell$, for a.e.\ $(t,s)\in\Delta^\comp[0,\infty)$; it holds that
\begin{equation*}
	\int^\infty_0e^{-2\mu t}\Bigl(\int^t_0|b(t,s,0,0)|\rd s\Bigr)^2\rd t+\int^\infty_0e^{-2\mu t}\int^t_0|\sigma(t,s,0,0)|^2\rd s\rd t<\infty.
\end{equation*}
\item[(iii)]
$h:[0,\infty)\times\bR^n\times\bR^\ell\to\bR$ is measurable; $h(t,x,u)$ is continuously differentiable in $(x,u)\in\bR^n\times\bR^\ell$ for a.e.\ $t\geq0$; there exists a constant $C>0$ such that
\begin{equation*}
	|h(t,x,u)|\leq C(1+|x|^2+|u|^2),\ |\partial_xh(t,x,u)|\leq C(1+|x|+|u|),\ |\partial_uh(t,x,u)|\leq C(1+|x|+|u|),
\end{equation*}
for any $(x,u)\in\bR^n\times\bR^\ell$, for a.e.\ $t\geq0$.
\end{itemize}
\end{assum}

We define
\begin{equation}\label{SVIE: criterion}
	\rho_{b,\sigma;x,u}:=\inf\{\rho\in\bR_+\,|\,[K_{b,u}]_1(\rho)+[K_{\sigma,u}]_2(\rho)<\infty,\ [K_{b,x}]_1(\rho)+[K_{\sigma,x}]_2(\rho)\leq1\}.
\end{equation}If $\mu>\rho_{b,\sigma;x,u}$ and $\lambda\geq2\mu$, the state process $X^u(\cdot)\in L^{2,-\mu}_\bF(0,\infty;\bR^n)$ and the discounted cost functional $J_\lambda(u(\cdot))\in\bR$ is well-defined for any $u(\cdot)\in\cU_{-\mu}$ (see Proposition 2.3 and Remark 2.4 in \cite{Ha21+}). We note that the parameters $\mu$ and $\lambda$ satisfying these conditions are strictly positive.

For a fixed control process $\hat{u}(\cdot)\in\cU_{-\mu}$, denote the corresponding state process by $\hat{X}(\cdot):=X^{\hat{u}}(\cdot)$. We use the following notations:
\begin{align*}
	&b_x(t,s):=\partial_xb(t,s,\hat{X}(s),\hat{u}(s)),\ b_u(t,s):=\partial_ub(t,s,\hat{X}(s),\hat{u}(s)),\\
	&\sigma^k_x(t,s):=\partial_x\sigma^k(t,s,\hat{X}(s),\hat{u}(s)),\ \sigma^k_u(t,s):=\partial_u\sigma^k(t,s,\hat{X}(s),\hat{u}(s)),\ k=1,\dots,d,
\end{align*}
for $(t,s)\in\Delta^\comp[0,\infty)$, and
\begin{equation*}
	h_x(t):=\partial_xh(t,\hat{X}(t),\hat{u}(t)),\ h_u(t):=\partial_uh(t,\hat{X}(t),\hat{u}(t)),
\end{equation*}
for $t\geq0$. We introduce the \emph{adjoint equation} of the following form:
\begin{equation}\label{SVIE: AE}
	\hat{Y}(t)=h_x(t)+\int^\infty_te^{-\lambda(s-t)}\Bigl\{b_x(s,t)^\top\hat{Y}(s)+\sum^d_{k=1}\sigma^k_x(s,t)^\top\hat{Z}^k(s,t)\Bigr\}\rd s-\int^\infty_t\hat{Z}(t,s)\rd W(s),\ t\geq0.
\end{equation}
The above equation is an infinite horizon BSVIE. As in \cite{Yo06,Yo08,Ha21+}, we say that a pair $(\hat{Y}(\cdot),\hat{Z}(\cdot,\cdot))$ is an \emph{adapted M-solution} to the infinite horizon BSVIE~\eqref{SVIE: AE} if it is in $\cM^{2,*}_\bF(0,\infty;\bR^n\times\bR^{n\times d})$ and satisfies the equation for a.e.\ $t\geq0$, a.s. By Theorem~3.7 in \cite{Ha21+}, we see that \eqref{SVIE: AE} admits a unique adapted M-solution $(\hat{Y}(\cdot),\hat{Z}(\cdot,\cdot))\in\cM^{2,-\mu}_\bF(0,\infty;\bR^n\times\bR^{n\times d})$.

We introduce a \emph{Hamiltonian functional} defined by
\begin{equation*}
	H_\lambda(t,x,u,p(\cdot),q(\cdot)):=h(t,x,u)+\int^\infty_te^{-\lambda(s-t)}\Bigl\{\langle b(s,t,x,u),p(s)\rangle+\sum^d_{k=1}\langle\sigma^k(s,t,x,u),q^k(s)\rangle\Bigr\}\rd s
\end{equation*}
for $(t,x,u,p(\cdot),q(\cdot))\in[0,\infty)\times\bR^n\times U\times L^{2,-\mu}(0,\infty;\bR^n)\times L^{2,-\mu}(0,\infty;\bR^{n\times d})$. By taking the conditional expectations $\bE_t[\cdot]$ on both sides of \eqref{SVIE: AE}, we see that the adjoint equation \eqref{SVIE: AE} can be written in terms of the Hamiltonian functional by
\begin{equation*}
	\hat{Y}(t)=\partial_xH_\lambda\bigl(t,\hat{X}(t),\hat{u}(t),\bE_t\bigl[\hat{Y}(\cdot)\bigr],\hat{Z}(\cdot,t)\bigr),\ \text{for a.e.}\ t\geq0,\ \text{a.s.}
\end{equation*}
Indeed, the above equation determines $\hat{Y}(t)$ and $\hat{Z}(t,s)$ for $(t,s)\in\Delta^\comp[0,\infty)$, and the term $\hat{Z}(t,s)$, $(t,s)\in\Delta[0,\infty)$, satisfying \eqref{SVIE: AE} is uniquely determined by them through the usual martingale representation theorem.

Under the above notations, the necessary and sufficient maximum principles for discounted control problems of SVIEs are stated as follows.

\begin{prop}[The necessary maximum principle for SVIEs; Theorem~4.3 in \cite{Ha21+}]\label{SVIE: prop_nMP}
Let \cref{SVIE: assum} hold, and suppose that $\mu,\lambda\in\bR$ satisfy $\mu>\rho_{b,\sigma;x,u}$ and $\lambda\geq2\mu$. If $\hat{u}(\cdot)\in\cU_{-\mu}$ is an optimal control, then the following \emph{optimality condition} holds for a.e.\ $t\geq0$ and a.s.:
\begin{equation}\label{SVIE: OC}
	\bigl\langle\partial_uH_\lambda\bigl(t,\hat{X}(t),\hat{u}(t),\bE_t\bigl[\hat{Y}(\cdot)\bigr],\hat{Z}(\cdot,t)\bigr),u-\hat{u}(t)\bigr\rangle\geq0,\ \forall\,u\in U.
\end{equation}
\end{prop}

\begin{prop}[The sufficient maximum principle for SVIEs; Theorem~4.5 in \cite{Ha21+}]\label{SVIE: prop_sMP}
Let \cref{SVIE: assum} hold, and suppose that $\mu>\rho_{b,\sigma;x,u}$ and $\lambda\geq2\mu$. Let $\hat{u}(\cdot)\in\cU_{-\mu}$ be given. Assume that the function
\begin{equation}\label{SVIE: convex}
	\bR^n\times U\ni(x,u)\mapsto H_\lambda\bigl(t,x,u,\bE_t\bigl[\hat{Y}(\cdot)\bigr],\hat{Z}(\cdot,t)\bigr)\in\bR
\end{equation}
is convex for a.e.\ $t\geq0$, a.s. Furthermore, assume that the optimality condition \eqref{SVIE: OC} holds for a.e.\ $t\geq0$, a.s. Then $\hat{u}(\cdot)$ is optimal.
\end{prop}

%% Reamrk

\begin{rem}\label{SVIE: rem_large}
Let $\mu>\rho_{b,\sigma;x,u}$ and $\lambda\geq2\mu$, fix a control process $\hat{u}(\cdot)\in\cU_{-\mu}$, and assume that the function \eqref{SVIE: convex} is convex for a.e.\ $t\geq0$, a.s. By Propositions~\ref{SVIE: prop_nMP} and \ref{SVIE: prop_sMP}, $\hat{u}(\cdot)\in\cU_{-\mu}$ is optimal over all control processes in $\cU_{-\mu}$ if and only if the optimality condition \eqref{SVIE: OC} holds for a.e.\ $t\geq0$, a.s. However, the latter does not depend on the choice of $\mu$. Therefore, in this case, the optimality of $\hat{u}(\cdot)\in\cU_{-\mu}$ on the control space $\cU_{-\mu}$ implies the optimality on the larger control space $\cU_{-\lambda/2}$, that is,
\begin{equation*}
	J_\lambda(\hat{u}(\cdot))=\inf_{u(\cdot)\in\cU_{-\lambda/2}}J_\lambda(u(\cdot)).
\end{equation*}
\end{rem}

We remark that, compared with the classical results on the maximum principle (see the textbook \cite{YoZh99}), all the above results (including the existence and uniqueness of the adapted M-solutions to infinite horizon BSVIEs) are proved without It\^{o}'s formula or BSDE theory. This is an interesting fact in a theoretical point of view.

%%%%%%%%%%%%%%%%
%%% Section
%%%%%%%%%%%%%%%%

\section{Discounted control problems for SDVIEs}\label{sec: SDVIE}

Next, we turn to the main topic of this paper. In this section, we consider discounted control problems of general stochastic delay Volterra integral equations (SDVIEs, for short). We obtain the necessary and sufficient maximum principles for this problem (see \cref{SDVIE: theo_MP} below). The resulting adjoint equation is a new kind of infinite horizon \emph{anticipated BSVIEs} (ABSVIEs, for short) which is beyond the class of ABSVIEs discussed in \cite{WeSh20}.

Suppose that the state process $X^u(\cdot)$ solves the controlled SDVIE~\eqref{1: state}, which we rewrite for readers' convenience:
\begin{equation*}
	\begin{dcases}
	X^u(t)=\varphi(t)+\int^t_0b(t,s,X^u(s),X^u(s-\delta),u(s))\rd s+\int^t_0\sigma(t,s,X^u(s),X^u(s-\delta),u(s))\rd W(s),\ t\geq0,\\
	X^u(t)=\varphi(t),\ t\in[-\delta,0].
	\end{dcases}
\end{equation*}
Here, $\delta\geq0$ is a given constant, $\varphi(\cdot)$ is a given free term, and $b:\Delta^\comp[0,\infty)\times\bR^n\times\bR^n\times\bR^\ell\to\bR^n$ and $\sigma:\Delta^\comp[0,\infty)\times\bR^n\times\bR^n\times\bR^\ell\to\bR^{n\times d}$ are given measurable maps. We consider the cost functional~\eqref{1: cost}:
\begin{equation*}
	J_\lambda(u(\cdot)):=\bE\Bigl[\int^\infty_0e^{-\lambda t}h(t,X^u(t),X^u(t-\delta),u(t))\rd t\Bigr],
\end{equation*}
where $\lambda\in\bR$ is a given discount rate and $h:[0,\infty)\times\bR^n\times\bR^n\times\bR^\ell\to\bR$ is a given measurable function. Our problem is to minimize the discounted cost functional \eqref{1: cost} over all control processes $u(\cdot)\in\cU_{-\mu}$, subject to the state equation given by the controlled SDVIE~\eqref{1: state}. We impose the following assumptions on the coefficients.

%% Assumption

\begin{assum}\label{SDVIE: assum}
\begin{itemize}
\item[(i)]
$(\varphi(t))_{t\in[-\delta,0]}$ is a deterministic continuous function on $[-\delta,0]$ with values in $\bR^n$, and $(\varphi(t))_{t\geq0}\in L^{2,-\mu}_\bF(0,\infty;\bR^n)$.
\item[(ii)]
$b(t,s,x_1,x_2,u)$ and $\sigma(t,s,x_1,x_2,u)$ are continuously differentiable with respect to $(x_1,x_2,u)\in\bR^n\times\bR^n\times\bR^\ell$ for a.e.\ $(t,s)\in\Delta^\comp[0,\infty)$; there exist $K_{b,x_1},K_{b,x_2},K_{b,u}\in L^{1,*}(0,\infty;\bR_+)$ and $K_{\sigma,x_1},K_{\sigma,x_2},K_{\sigma,u}\in L^{2,*}(0,\infty;\bR_+)$ such that
\begin{align*}
	&|b(t,s,x_1,x_2,u)\mathalpha{-}b(t,s,x'_1,x'_2,u')|\leq K_{b,x_1}(t\mathalpha{-}s)|x_1\mathalpha{-}x'_1|+K_{b,x_2}(t\mathalpha{-}s)|x_2\mathalpha{-}x'_2|+K_{b,u}(t\mathalpha{-}s)|u\mathalpha{-}u'|,\\
	&|\sigma(t,s,x_1,x_2,u)\mathalpha{-}\sigma(t,s,x'_1,x'_2,u')|\leq K_{\sigma,x_1}(t\mathalpha{-}s)|x_1\mathalpha{-}x'_1|+K_{\sigma,x_2}(t\mathalpha{-}s)|x_2\mathalpha{-}x'_2|+K_{\sigma,u}(t\mathalpha{-}s)|u\mathalpha{-}u'|,
\end{align*}
for any $x_1,x'_1,x_2,x'_2\in\bR^n$ and $u,u'\in\bR^\ell$, for a.e.\ $(t,s)\in\Delta^\comp[0,\infty)$; it holds that
\begin{equation*}
	\int^\infty_0e^{-2\mu t}\Bigl(\int^t_0|b(t,s,0,0,0)|\rd s\Bigr)^2\rd t+\int^\infty_0e^{-2\mu t}\int^t_0|\sigma(t,s,0,0,0)|^2\rd s\rd t<\infty.
\end{equation*}
\item[(iii)]
$h(t,x_1,x_2,u)$ is continuously differentiable in $(x_1,x_2,u)\in\bR^n\times\bR^n\times\bR^\ell$ for a.e.\ $t\geq0$; there exists a constant $C>0$ such that
\begin{equation*}
	|h(t,x_1,x_2,u)|\leq C(1+|x_1|^2+|x_2|^2+|u|^2),\ |\partial_\xi h(t,x_1,x_2,u)|\leq C(1+|x_1|+|x_2|+|u|),\ \xi=x_1,x_2,u,
\end{equation*}
for any $(x_1,x_2,u)\in\bR^n\times\bR^n\times\bR^\ell$, for a.e.\ $t\geq0$.
\end{itemize}
\end{assum}

Now we convert the controlled SDVIE~\eqref{1: state} into an auxiliary controlled SVIE (without delay) by using a simple idea to ``lift up'' the state process. Define an $\bR^{2n}$-valued auxiliary state process $\bX^u(\cdot)$ by
\begin{equation*}
	\bX^u(\cdot):=\left(\begin{array}{c}X^u_1(\cdot)\\X^u_2(\cdot)\end{array}\right),
\end{equation*}
where $X^u_1(t):=X^u(t)$ and $X^u_2(t):=X^u(t-\delta)$ for $t\geq0$. From the SDVIE~\eqref{1: state}, we see that
\begin{equation*}
	X^u_1(t)=\varphi(t)+\int^t_0b(t,s,X^u_1(s),X^u_2(s),u(s))\rd s+\int^t_0\sigma(t,s,X^u_1(s),X^u_2(s),u(s))\rd W(s)
\end{equation*}
and
\begin{equation*}
	X^u_2(t)=X^u(t-\delta)=
	\begin{dcases}
		\varphi(t-\delta)\ \text{if}\ t\in[0,\delta],\\
		\varphi(t-\delta)+\int^{t-\delta}_0b(t-\delta,s,X^u_1(s),X^u_2(s),u(s))\rd s\\
		\hspace{1cm}+\int^{t-\delta}_0\sigma(t-\delta,s,X^u_1(s),X^u_2(s),u(s))\rd W(s),\ \text{if}\ t\geq\delta,
	\end{dcases}
\end{equation*}
for $t\geq0$. The latter can be written as
\begin{align*}
	X^u_2(t)&=\varphi(t-\delta)+\int^t_0\1_{[\delta,\infty)}(t-s)b(t-\delta,s,X^u_1(s),X^u_2(s),u(s))\rd s\\
	&\hspace{3cm}+\int^t_0\1_{[\delta,\infty)}(t-s)\sigma(t-\delta,s,X^u_1(s),X^u_2(s),u(s))\rd W(s),\ t\geq0.
\end{align*}
Therefore, we see that the dynamics of the auxiliary state process $\bX^u(\cdot)$ is described by the following SVIE:
\begin{equation}\label{SDVIE: aux_state}
	\bX^u(t)=\tilde{\varphi}(t)+\int^t_0\tilde{b}(t,s,\bX^u(s),u(s))\rd s+\int^t_0\tilde{\sigma}(t,s,\bX^u(s),u(s))\rd W(s),\ t\geq0,
\end{equation}
where
\begin{equation*}
	\tilde{\varphi}(t):=\left(\begin{array}{c}\varphi(t)\\\varphi(t-\delta)\end{array}\right)
\end{equation*}
for $t\geq0$ and, with the notations $J_1:=(I_{n\times n},0_{n\times n})$ and $J_2:=(0_{n\times n},I_{n\times n})$,
\begin{align*}
	&\tilde{b}(t,s,\bX,u):=\left(\begin{array}{c}b(t,s,J_1\bX,J_2\bX,u)\\\1_{[\delta,\infty)}(t-s)b(t-\delta,s,J_1\bX,J_2\bX,u)\end{array}\right),\\
	&\tilde{\sigma}(t,s,\bX,u):=\left(\begin{array}{c}\sigma(t,s,J_1\bX,J_2\bX,u)\\\1_{[\delta,\infty)}(t-s)\sigma(t-\delta,s,J_1\bX,J_2\bX,u)\end{array}\right),
\end{align*}
for $(t,s)\in\Delta^\comp[0,\infty)$, $\bX\in\bR^{2n}$, and $u\in \bR^\ell$. Also, the cost functional is written by
\begin{equation}\label{SDVIE: aux_cost}
	J_\lambda(u(\cdot))=\bE\Bigl[\int^\infty_0e^{-\lambda t}\tilde{h}(t,\bX^u(t),u(t))\rd t\Bigr],
\end{equation}
where $\tilde{h}:[0,\infty)\times\bR^{2n}\times\bR^\ell\to\bR$ is defined by $\tilde{h}(t,\bX,u):=h(t,J_1\bX,J_2\bX,u)$ for $(t,\bX,u)\in[0,\infty)\times\bR^{2n}\times\bR^\ell$. It is easy to see that there exists a unique solution $X^u(\cdot)\in L^{2,-\mu}_{\bF}(0,\infty;\bR^n)$ of SDVIE~\eqref{1: state} if and only if there exists a unique solution $\bX^u(\cdot)\in L^{2,-\mu}_\bF(0,\infty;\bR^{2n})$ of SVIE~\eqref{SDVIE: aux_state}, and the solutions are connected via the relations $J_1\bX^u(t)=X^u_1(t)=X^u(t)$ and $J_2\bX^u(t)=X^u_2(t)=X^u(t-\delta)$ for $t\geq0$. Observe that, for $f=b,\sigma$,
\begin{align*}
	&|\tilde{f}(t,s,\bX,u)-\tilde{f}(t,s,\bX',u')|\\
	&\leq|f(t,s,J_1\bX,J_2\bX,u)-f(t,s,J_1\bX',J_2\bX',u')|\\
	&\hspace{1cm}+\1_{[\delta,\infty)}(t-s)|f(t-\delta,s,J_1\bX,J_2\bX,u)-f(t-\delta,s,J_1\bX',J_2\bX',u')|\\
	&\leq K_{\tilde{f},\bX}(t-s)|\bX-\bX'|+K_{\tilde{f},u}(t-s)|u-u'|
\end{align*}
for any $\bX,\bX'\in\bR^{2n}$ and $u,u'\in\bR^\ell$, for a.e.\ $(t,s)\in\Delta^\comp[0,\infty)$. Here, $K_{\tilde{b},\bX},K_{\tilde{b},u}\in L^{1,*}(0,\infty;\bR_+)$ and $K_{\tilde{\sigma},\bX},K_{\tilde{\sigma},u}\in L^{2,*}(0,\infty;\bR_+)$ are defined by
\begin{align*}
	&K_{\tilde{f},\bX}(\tau):=K_{f,x_1}(\tau)+K_{f,x_2}(\tau)+\1_{[\delta,\infty)}(\tau)\bigl\{K_{f,x_1}(\tau-\delta)+K_{f,x_2}(\tau-\delta)\bigr\},\\
	&K_{\tilde{f},u}(\tau):=K_{f,u}(\tau)+\1_{[\delta,\infty)}(\tau)K_{f,u}(\tau-\delta),
\end{align*}
for $\tau\geq0$ and $f=b,\sigma$. Under the above notations, we have
\begin{align*}
	[K_{\tilde{b},\bX}]_1(\rho)&=\int^\infty_0e^{-\rho\tau}K_{\tilde{b},\bX}(\tau)\rd\tau\\
	&=\int^\infty_0e^{-\rho\tau}K_{b,x_1}(\tau)\rd\tau+\int^\infty_0e^{-\rho\tau}K_{b,x_2}(\tau)\rd\tau\\
	&\hspace{1cm}+\int^\infty_\delta e^{-\rho\tau}K_{b,x_1}(\tau-\delta)\rd\tau+\int^\infty_\delta e^{-\rho\tau}K_{b,x_2}(\tau-\delta)\rd\tau\\
	&=(1+e^{-\rho\delta})([K_{b,x_1}]_1(\rho)+[K_{b,x_2}]_1(\rho))
\end{align*}
and
\begin{align*}
	[K_{\tilde{\sigma},\bX}]_2(\rho)&=\Bigl(\int^\infty_0e^{-2\rho\tau}K_{\tilde{\sigma},\bX}(\tau)^2\rd\tau\Bigr)^{1/2}\\
	&\leq\Bigl(\int^\infty_0e^{-2\rho\tau}K_{\sigma,x_1}(\tau)^2\rd\tau\Bigr)^{1/2}+\Bigl(\int^\infty_0e^{-2\rho\tau}K_{\sigma,x_2}(\tau)^2\rd\tau\Bigr)^{1/2}\\
	&\hspace{1cm}+\Bigl(\int^\infty_\delta e^{-2\rho\tau}K_{\sigma,x_1}(\tau-\delta)^2\rd\tau\Bigr)^{1/2}+\Bigl(\int^\infty_\delta e^{-2\rho\tau}K_{\sigma,x_2}(\tau-\delta)^2\rd\tau\Bigr)^{1/2}\\
	&=(1+e^{-\rho\delta})([K_{\sigma,x_1}]_2(\rho)+[K_{\sigma,x_2}]_2(\rho))
\end{align*}
for $\rho\in\bR$. Similarly, we have
\begin{equation*}
	[K_{\tilde{b},u}]_1(\rho)=(1+e^{-\rho\delta})[K_{b,u}]_1(\rho)\ \text{and}\ [K_{\tilde{\sigma},u}]_2(\rho)\leq(1+e^{-\rho\delta})[K_{\sigma,u}]_2(\rho),\ \rho\in\bR.
\end{equation*}
Therefore, as to the constant \eqref{SVIE: criterion} for the auxiliary SVIE~\eqref{SDVIE: aux_state}, we can take
\begin{equation}\label{SDVIE: criterion}
	\rho_{b,\sigma;x_1,x_2,u}:=\inf\left\{\rho\in\bR_+\relmiddle|\begin{aligned}&[K_{b,u}]_1(\rho)+[K_{\sigma,u}]_2(\rho)<\infty,\\&(1+e^{-\rho\delta})\bigl\{[K_{b,x_1}]_1(\rho)+[K_{b,x_2}]_1(\rho)+[K_{\sigma,x_1}]_2(\rho)+[K_{\sigma,x_2}]_2(\rho)\bigr\}\leq1\end{aligned}\right\}.
\end{equation}
Suppose that $\mu$ and $\lambda$ satisfy $\mu>\rho_{b,\sigma;x_1,x_2,u}$ and $\lambda\geq2\mu$. Then by Proposition~2.3 in \cite{Ha21+}, for any $u(\cdot)\in\cU_{-\mu}$, there exists a unique solution $\bX^u(\cdot)\in L^{2,-\mu}_\bF(0,\infty;\bR^{2n})$ to the auxiliary SVIE~\eqref{SDVIE: aux_state}, and the auxiliary cost functional \eqref{SDVIE: aux_cost} is well-defined. Equivalently, there exists a unique solution $X^u(\cdot)\in L^{2,-\mu}_\bF(0,\infty;\bR^n)$ to the controlled SDVIE~\eqref{1: state}, and the original cost functional~\eqref{1: cost} is well-defined.

We apply the results in \cref{sec: SVIE} to the auxiliary control problem \eqref{SDVIE: aux_state}--\eqref{SDVIE: aux_cost}. The auxiliary Hamiltonian functional $\tilde{H}_\lambda:[0,\infty)\times\bR^{2n}\times U\times L^{2,-\mu}(0,\infty;\bR^{2n})\times L^{2,-\mu}(0,\infty;\bR^{(2n)\times d})\to\bR$ becomes
\begin{align*}
	&\tilde{H}_\lambda(t,\bX,u,p(\cdot),q(\cdot))=\tilde{h}(t,\bX,u)+\int^\infty_te^{-\lambda(s-t)}\Bigl\{\langle \tilde{b}(s,t,\bX,u),p(s)\rangle+\sum^d_{k=1}\langle\tilde{\sigma}^k(s,t,\bX,u),q^k(s)\rangle\Bigr\}\rd s\\
	&=h(t,J_1\bX,J_2\bX,u)+\int^\infty_te^{-\lambda(s-t)}\{\langle b(s,t,J_1\bX,J_2\bX,u),J_1p(s)\rangle\\
	&\hspace{6cm}+\1_{[\delta,\infty)}(s-t)\langle b(s-\delta,t,J_1\bX,J_2\bX,u),J_2p(s)\rangle\}\rd s\\
	&\hspace{1cm}+\sum^d_{k=1}\int^\infty_te^{-\lambda(s-t)}\{\langle\sigma^k(s,t,J_1\bX,J_2\bX,u),J_1q^k(s)\rangle\\
	&\hspace{5cm}+\1_{[\delta,\infty)}(s-t)\langle\sigma^k(s-\delta,t,J_1\bX,J_2\bX,u),J_2q^k(s)\rangle\}\rd s\\
	&=h(t,J_1\bX,J_2\bX,u)+\int^\infty_te^{-\lambda(s-t)}\langle b(s,t,J_1\bX,J_2\bX,u),J_1p(s)+e^{-\lambda\delta}J_2p(s+\delta)\rangle\rd s\\
	&\hspace{1cm}+\sum^d_{k=1}\int^\infty_te^{-\lambda(s-t)}\langle\sigma^k(s,t,J_1\bX,J_2\bX,u),J_1q^k(s)+e^{-\lambda\delta}J_2q^k(s+\delta)\rangle\rd s,
\end{align*}
where in the last equality we used the change of variable formula. Fix an arbitrary control process $\hat{u}(\cdot)\in\cU_{-\mu}$, and denote the corresponding state process (resp., auxiliary state process) by $\hat{X}(\cdot):=X^{\hat{u}}(\cdot)$ (resp., $\hat{\bX}(\cdot):=\bX^{\hat{u}}(\cdot)$). The auxiliary adjoint equation
\begin{equation}\label{SDVIE: aux_AE}
	\hat{\bY}(t)=\partial_\bX\tilde{H}_\lambda(t,\hat{\bX}(t),\hat{u}(t),\bE_t[\hat{\bY}(\cdot)],\hat{\bZ}(\cdot,t)),\ t\geq0,
\end{equation}
admits a unique adapted M-solution $(\hat{\bY}(\cdot),\hat{\bZ}(\cdot,\cdot))\in\cM^{2,-\mu}_\bF(0,\infty;\bR^{2n}\times\bR^{(2n)\times d})$. Also, the optimality condition for the auxiliary problem \eqref{SDVIE: aux_state}--\eqref{SDVIE: aux_cost} is
\begin{equation}\label{SDVIE: aux_OC}
	\langle\partial_u\tilde{H}_\lambda(t,\hat{\bX}(t),\hat{u}(t),\bE_t[\hat{\bY}(\cdot)],\hat{\bZ}(\cdot,t)),u-\hat{u}(t)\rangle\geq0,\ \forall\,u\in U.
\end{equation}

Now we further proceed the above investigation and obtain the necessary and sufficient maximum principles for the original problem \eqref{1: state}--\eqref{1: cost}. For $i=1,2$, defining $\hat{Y}_i(\cdot):=J_i\hat{\bY}(\cdot)$ and $\hat{Z}_i(\cdot,\cdot):=J_i\hat{\bZ}(\cdot,\cdot)$, we see that $(\hat{Y}_i(\cdot),\hat{Z}_i(\cdot,\cdot))$ is in $\cM^{2,-\mu}_\bF(0,\infty;\bR^n\times\bR^{n\times d})$ and solves the following equation:
\begin{equation}\label{SDVIE: AE_component}
\begin{split}
	\hat{Y}_i(t)&=h_{x_i}(t)+\int^\infty_te^{-\lambda(s-t)}b_{x_i}(s,t)^\top\bE_t[\hat{Y}_1(s)+e^{-\lambda\delta}\hat{Y}_2(s+\delta)]\rd s\\
	&\hspace{1cm}+\sum^d_{k=1}\int^\infty_te^{-\lambda(s-t)}\sigma^k_{x_i}(s,t)^\top\{\hat{Z}^k_1(s,t)+e^{-\lambda\delta}\hat{Z}^k_2(s+\delta,t)\}\rd s,\ t\geq0,
\end{split}
\end{equation}
where $h_{x_i}(t):=\partial_{x_i}h(t,J_1\hat{\bX}(t),J_2\hat{\bX}(t),\hat{u}(t))=\partial_{x_i}h(t,\hat{X}(t),\hat{X}(t-\delta),\hat{u}(t))$, $t\geq0$, and
\begin{equation*}
	f_{x_i}(t,s):=\partial_{x_i}f(t,s,J_1\hat{\bX}(s),J_2\hat{\bX}(s),\hat{u}(s))=\partial_{x_i}f(t,s,\hat{X}(s),\hat{X}(s-\delta),\hat{u}(s)),\ (t,s)\in\Delta^\comp[0,\infty),
\end{equation*}
for $f=b,\sigma^k$, $k=1,\dots,d$. Define
\begin{equation*}
	\hat{Y}(t):=\hat{Y}_1(t)+e^{-\lambda\delta}\bE_t[\hat{Y}_2(t+\delta)]\ \text{and}\ \hat{Z}(t,s):=\hat{Z}_1(t,s)+e^{-\lambda\delta}\hat{Z}_2(t+\delta,s)
\end{equation*}
for $(t,s)\in[0,\infty)^2$. Since $(\hat{Y}_i(\cdot),\hat{Z}_i(\cdot,\cdot))\in\cM^{2,-\mu}_\bF(0,\infty;\bR^n\times\bR^{n\times d})$, we have
\begin{align*}
	\hat{Y}(t)&=\bE[\hat{Y}_1(t)]+\int^t_0\hat{Z}_1(t,s)\rd W(s)+e^{-\lambda\delta}\Bigl\{\bE[\hat{Y}_2(t+\delta)]+\int^t_0\hat{Z}_2(t+\delta,s)\rd W(s)\Bigr\}\\
	&=\bE[\hat{Y}(t)]+\int^t_0\hat{Z}(t,s)\rd W(s)
\end{align*}
for a.e.\ $t\geq0$, a.s., which implies that $(\hat{Y}(\cdot),\hat{Z}(\cdot,\cdot))$ is in $\cM^{2,-\mu}_\bF(0,\infty;\bR^n\times\bR^{n\times d})$. Furthermore, by \eqref{SDVIE: AE_component}, it holds that
\begin{align*}
	\hat{Y}(t)&=h_{x_1}(t)+\int^\infty_te^{-\lambda(s-t)}\Bigl\{b_{x_1}(s,t)^\top\bE_t[\hat{Y}(s)]+\sum^d_{k=1}\sigma^k_{x_1}(s,t)^\top\hat{Z}^k(s,t)\Bigr\}\rd s\\
	&\hspace{0.3cm}+e^{-\lambda\delta}\bE_t\Bigl[h_{x_2}(t+\delta)+\int^\infty_{t+\delta}e^{-\lambda(s-t-\delta)}\Bigl\{b_{x_2}(s,t+\delta)^\top\bE_{t+\delta}[\hat{Y}(s)]+\sum^d_{k=1}\sigma^k_{x_2}(s,t+\delta)^\top\hat{Z}^k(s,t+\delta)\Bigr\}\rd s\Bigr]\\
	&=\bE_t\Bigl[h_{x_1}(t)+e^{-\lambda\delta}h_{x_2}(t+\delta)+\int^\infty_te^{-\lambda(s-t)}\Bigl\{b_{x_1}(s,t)^\top\hat{Y}(s)+\sum^d_{k=1}\sigma^k_{x_1}(s,t)^\top\hat{Z}^k(s,t)\\
	&\hspace{3cm}+e^{-\lambda\delta}\bE_s\Bigl[b_{x_2}(s+\delta,t+\delta)^\top\hat{Y}(s+\delta)+\sum^d_{k=1}\sigma^k_{x_2}(s+\delta,t+\delta)^\top\hat{Z}^k(s+\delta,t+\delta)\Bigr]\Bigr\}\rd s\Bigr].
\end{align*}
By redefining $\hat{Z}(t,s)$ for $(t,s)\in\Delta[0,\infty)$ (which is uniquely determined by $\hat{Y}(t)$ and $\hat{Z}(t,s)$ for $(t,s)\in\Delta^\comp[0,\infty)$), we see that $(\hat{Y}(\cdot),\hat{Z}(\cdot,\cdot))\in\cM^{2,-\mu}_\bF(0,\infty;\bR^n\times\bR^{n\times d})$ solves the following infinite horizon \emph{anticipated BSVIE} (ABSVIE, for short):
\begin{equation}\label{SDVIE: AE}
\begin{split}
	\hat{Y}(t)&=h_{x_1}(t)+e^{-\lambda\delta}h_{x_2}(t+\delta)+\int^\infty_te^{-\lambda(s-t)}\Bigl\{b_{x_1}(s,t)^\top\hat{Y}(s)+\sum^d_{k=1}\sigma^k_{x_1}(s,t)^\top\hat{Z}^k(s,t)\\
	&\hspace{2cm}+e^{-\lambda\delta}\bE_s\Bigl[b_{x_2}(s+\delta,t+\delta)^\top\hat{Y}(s+\delta)+\sum^d_{k=1}\sigma^k_{x_2}(s+\delta,t+\delta)^\top\hat{Z}^k(s+\delta,t+\delta)\Bigr]\Bigr\}\rd s\\
	&\hspace{1cm}-\int^\infty_t\hat{Z}(t,s)\rd W(s),\ t\geq0.
\end{split}
\end{equation}
Conversely, if $(\hat{Y}(\cdot),\hat{Z}(\cdot,\cdot))\in\cM^{2,-\mu}_\bF(0,\infty;\bR^n\times\bR^{n\times d})$ satisfies the infinite horizon ABSVIE \eqref{SDVIE: AE}, then the pair $(\hat{\bY}(\cdot),\hat{\bZ}(\cdot,\cdot))\in\cM^{2,-\mu}_\bF(0,\infty;\bR^{2n}\times\bR^{(2n)\times d})$ defined by
\begin{align*}
	J_i\hat{\bY}(t):=h_{x_i}(t)+\int^\infty_te^{-\lambda(s-t)}\Bigl\{b_{x_i}(s,t)^\top\bE_t[\hat{Y}(s)]+\sum^d_{k=1}\sigma^k_{x_i}(s,t)^\top\hat{Z}^k(s,t)\Bigr\}\rd s,\ t\geq0,\ i=1,2,
\end{align*}
and $\hat{\bY}(t)=\bE[\hat{\bY}(t)]+\int^t_0\hat{\bZ}(t,s)\rd W(s)$, $t\geq0$, satisfies the relations
\begin{equation}\label{SDVIE: ABSVIE_BSVIE}
	\hat{Y}(t)=J_1\hat{\bY}(t)+e^{-\lambda\delta}\bE_t[J_2\hat{\bY}(t+\delta)]\ \text{and}\ \hat{Z}(t,s)=J_1\hat{\bZ}(t,s)+e^{-\lambda\delta}J_2\hat{\bZ}(t+\delta,s),
\end{equation}
for $(t,s)\in\Delta^\comp[0,\infty)$. Therefore, we see that $(\hat{\bY}(\cdot),\hat{\bZ}(\cdot,\cdot))\in\cM^{2,-\mu}_\bF(0,\infty;\bR^{2n}\times\bR^{(2n)\times d})$ is an adapted M-solution to the infinite horizon (standard) BSVIE \eqref{SDVIE: aux_AE}. By the existence and uniqueness of the adapted M-solution to the infinite horizon (standard) BSVIE (see Theorem~3.7 in \cite{Ha21+}), we see that the infinite horizon ABSVIE \eqref{SDVIE: AE} admits a unique adapted M-solution $(\hat{Y}(\cdot),\hat{Z}(\cdot,\cdot))\in\cM^{2,-\mu}_\bF(0,\infty;\bR^n\times\bR^{n\times d})$ given by \eqref{SDVIE: ABSVIE_BSVIE}. Furthermore, the optimality condition \eqref{SDVIE: aux_OC} becomes
\begin{equation}\label{SDVIE: OC}
	\Bigl\langle h_u(t)+\int^\infty_te^{-\lambda(s-t)}\Bigl\{b_u(s,t)^\top\bE_t[\hat{Y}(s)]+\sum^d_{k=1}\sigma^k_u(s,t)^\top\hat{Z}^k(s,t)\Bigr\}\rd s,u-\hat{u}(t)\Bigr\rangle\geq0,\ \forall\,u\in U,
\end{equation}
where $h_u(t):=\partial_uh(t,\hat{X}(t),\hat{X}(t-\delta),\hat{u}(t))$ and
\begin{equation*}
	f_u(t,s):=\partial_uf(t,s,\hat{X}(s),\hat{X}(s-\delta),\hat{u}(s)),\ (t,s)\in\Delta^\comp[0,\infty),\ f=b,\sigma^k,\ k=1,\dots,d.
\end{equation*}
Consequently, we obtain the following (necessary and sufficient) maximum principles for discounted control problems of SDVIEs.

%% Theorem

\begin{theo}[The necessary and sufficient maximum principles for SDVIEs]\label{SDVIE: theo_MP}
Suppose that the coefficients $b$, $\sigma$ and $h$ satisfy \cref{SDVIE: assum}. Let $\mu$ and $\lambda$ satisfy $\mu>\rho_{b,\sigma;x_1,x_2,u}$ and $\lambda\geq2\mu$, where $\rho_{b,\sigma;x_1,x_2,u}\in[0,\infty)$ is defined by \eqref{SDVIE: criterion}. Fix an arbitrary control process $\hat{u}(\cdot)\in\cU_{-\mu}$, and denote by $\hat{X}(\cdot):=X^{\hat{u}}(\cdot)$ the corresponding solution to the controlled SDVIE \eqref{1: state}. Also, let $(\hat{Y}(\cdot),\hat{Z}(\cdot,\cdot))\in \cM^{2,-\mu}_\bF(0,\infty;\bR^n\times\bR^{n\times d})$ be the unique adapted M-solution to the infinite horizon ABSVIE \eqref{SDVIE: AE}. Then the following hold:
\begin{itemize}
\item[(i)]
If $\hat{u}(\cdot)$ is an optimal control for the discounted control problem \eqref{1: state}--\eqref{1: cost}, then the optimality condition \eqref{SDVIE: OC} holds for a.e.\ $t\geq0$, a.s.
\item[(ii)]
If the map
\begin{align*}
	&\bR^n\times\bR^n\times U\ni(x_1,x_2,u)\mapsto h(t,x_1,x_2,u)\\
	&\hspace{1cm}+\int^\infty_te^{-\lambda(s-t)}\Bigl\{\langle b(s,t,x_1,x_2,u),\bE_t[\hat{Y}(s)]\rangle+\sum^d_{k=1}\langle\sigma^k(s,t,x_1,x_2,u),\hat{Z}^k(s,t)\rangle\Bigr\}\rd s\in\bR
\end{align*}
is convex for a.e.\ $t\geq0$, a.s., and if the optimality condition \eqref{SDVIE: OC} holds for a.e.\ $t\geq0$, a.s., then $\hat{u}(\cdot)$ is an optimal control for the discounted control problem \eqref{1: state}--\eqref{1: cost}.
\end{itemize}
\end{theo}

%% Remark

\begin{rem}\label{SDVIE: rem_ABSVIE}
The adjoint equation \eqref{SDVIE: AE} can be written as
\begin{equation}\label{SDVIE: ABSVIE}
	Y(t)=\psi(t)+\int^\infty_te^{-\lambda(s-t)}g(t,s,Y(s),Z(s,t),Y(s+\delta),Z(s+\delta,t+\delta))\rd s-\int^\infty_tZ(t,s)\rd W(s),\ t\geq0,
\end{equation}
for suitable free term $\psi(\cdot)$ and driver $g$. This is a novel class of (infinite horizon) ABSVIE, which is a BSVIE whose driver depends on the ``anticipated terms'' $Y(s+\delta)$ and $Z(s+\delta,t+\delta)$. Indeed, Wen and Shi~\cite{WeSh20} studied the well-posedness of (finite horizon) ABSVIEs of the following form:
\begin{equation*}
	\begin{dcases}
	Y(t)=\psi(t)+\int^T_tg(t,s,Y(s),Z(t,s),Z(s,t),Y(s+\delta),Z(t,s+\delta),Z(s+\delta,t))\rd s\\
	\hspace{2cm}-\int^T_tZ(t,s)\rd W(s),\ t\in[0,T],\\
	Y(t)=\psi(t),\ t\in[T,T+\delta],\ Z(t,s)=\eta(t,s),\ (t,s)\in[0,T+\delta]^2\setminus[0,T]^2,
	\end{dcases}
\end{equation*}
where $\psi(\cdot)$ and $\eta(\cdot,\cdot)$ are some given processes. Besides the fact that \eqref{SDVIE: ABSVIE} is defined on the infinite horizon, it is beyond the class of \cite{WeSh20} since the driver of our ABSVIE~\eqref{SDVIE: ABSVIE} depends on the new anticipated term $Z(s+\delta,t+\delta)$ for $(t,s)\in\Delta[0,\infty)$. \cref{SDVIE: theo_MP} shows that (infinite horizon) ABSVIEs of the form \eqref{SDVIE: ABSVIE} are appropriate tools to derive the maximum principle for SDVIEs. It is remarkable that the infinite horizon ABSVIE~\eqref{SDVIE: AE} can be obtained by an appropriate transformation of the infinite horizon (non-anticipated) BSVIE~\eqref{SDVIE: aux_AE} which is the adjoint equation for the auxiliary problem \eqref{SDVIE: aux_state}--\eqref{SDVIE: aux_cost} of an SVIE (without delay).
\end{rem}

%%%%%%%%%%%%%%%%
%%% Section
%%%%%%%%%%%%%%%%

\section{Discounted control problems for SDDEs}\label{sec: SDDE}

In this section, we apply the above results to discounted control problems of stochastic delay differential equations (SDDEs, for short). In finite horizon settings, similar problems were studied by Chen and Wu~\cite{ChWu10} and {\O}ksendal, Sulem, and Zhang~\cite{OkSuZh11}. However, our method deriving the maximum principle is different from \cite{ChWu10,OkSuZh11} since it is based on BSVIE theory discussed in Sections~\ref{sec: SVIE} and \ref{sec: SDVIE}.

Suppose that the state process $X^u(\cdot)$ solves the controlled SDDE~\eqref{1: SDDE_state}, which we rewrite for readers' convenience:
\begin{equation*}
	\begin{dcases}
	\mathrm{d}X^u(t)=b(t,X^u(t),X^u(t-\delta),u(t))\rd t+\sigma(t,X^u(t),X^u(t-\delta),u(t))\rd W(t),\ t\geq0,\\
	X^u(t)=\varphi(t),\ t\in[-\delta,0].
	\end{dcases}
\end{equation*}
Here, $\delta\geq0$ is a given constant, $(\varphi(t))_{t\in[-\delta,0]}$ is a given initial condition which is assumed to be an $\bR^n$-valued deterministic function, and $b:[0,\infty)\times\bR^n\times\bR^n\times\bR^\ell\to\bR^n$ and $\sigma:[0,\infty)\times\bR^n\times\bR^n\times\bR^\ell\to\bR^{n\times d}$ are given measurable maps. We impose the following assumptions on the coefficients.

%% Assumption

\begin{assum}\label{SDDE: assum}
\begin{itemize}
\item[(i)]
$(\varphi(t))_{t\in[-\delta,0]}$ is a deterministic continuous function on $[-\delta,0]$ with values in $\bR^n$.
\item[(ii)]
$b(s,x_1,x_2,u)$ and $\sigma(s,x_1,x_2,u)$ are continuously differentiable with respect to $(x_1,x_2,u)\in\bR^n\times\bR^n\times\bR^\ell$ for a.e.\ $s\geq0$; there exist six constants $L_{b,x_1},L_{b,x_2},L_{b,u},L_{\sigma,x_1},L_{\sigma,x_2},L_{\sigma,u}>0$ such that
\begin{align*}
	&|b(s,x_1,x_2,u)-b(s,x'_1,x'_2,u')|\leq L_{b,x_1}|x_1-x'_1|+L_{b,x_2}|x_2-x'_2|+L_{b,u}|u-u'|,\\
	&|\sigma(s,x_1,x_2,u)-\sigma(s,x'_1,x'_2,u')|\leq L_{\sigma,x_1}|x_1-x'_1|+L_{\sigma,x_2}|x_2-x'_2|+L_{\sigma,u}|u-u'|,
\end{align*}
for any $x_1,x'_1,x_2,x'_2\in\bR^n$ and $u,u'\in\bR^\ell$, for a.e.\ $s\geq0$; it holds that
\begin{equation*}
	\int^\infty_0e^{-\mu s}|b(s,0,0,0)|\rd s+\int^\infty_0e^{-2\mu s}|\sigma(s,0,0,0)|^2\rd s<\infty.
\end{equation*}
\end{itemize}
\end{assum}

Consider the cost functional \eqref{1: cost} with a given discount rate $\lambda\in\bR$ and a given measurable function $h:[0,\infty)\times\bR^n\times\bR^n\times\bR^\ell\to\bR$ satisfying the condition (iii) in \cref{SDVIE: assum}. Our problem is to minimize the discounted cost functional \eqref{1: cost} over all control processes $u(\cdot)\in\cU_{-\mu}$, subject to the state equation given by the controlled SDDE~\eqref{1: SDDE_state}.

The control problem \eqref{1: SDDE_state}--\eqref{1: cost} is a special case of the framework in \cref{sec: SDVIE}. In this case, we can take $K_{f,\xi}(\tau)=L_{f,\xi}$ for any $\tau\geq0$, $f=b,\sigma$, and $\xi=x_1,x_2,u$. Thus, simple calculations show that the constant \eqref{SDVIE: criterion} becomes $\rho_{b,\sigma;x_1,x_2,u}=\rho_1:=\inf\bigl\{\rho>0\,|\,(1+e^{-\rho\delta})\bigl(\frac{L_{b,x_1}+L_{b,x_2}}{\rho}+\frac{L_{\sigma,x_1}+L_{\sigma,x_2}}{\sqrt{2\rho}}\bigr)\leq1\bigr\}$. Suppose that the constants $\mu$ and $\lambda$ satisfy $\mu>\rho_1$ and $\lambda\geq2\mu$. Then for any $u(\cdot)\in\cU_{-\mu}$, there exists a unique solution $X^u(\cdot)\in L^{2,-\mu}_\bF(0,\infty;\bR^n)$ to the controlled SDDE~\eqref{1: SDDE_state}, and the cost functional \eqref{1: cost} is well-defined.

Fix an arbitrary control process $\hat{u}(\cdot)$, and denote the corresponding state process by $\hat{X}(\cdot):=X^{\hat{u}}(\cdot)$. In this section, we use the notations $f_{\xi}(t):=\partial_\xi f(t,\hat{X}(t),\hat{X}(t-\delta),\hat{u}(t))$ for each $t\geq0$, $\xi=x_1,x_2,u$, and $f=h,b,\sigma^k$, $k=1,\dots,d$. By taking the conditional expectations and applying the change of variable formula, we see that the infinite horizon ABSVIE~\eqref{SDVIE: AE} becomes
\begin{align*}
	\hat{Y}(t)&=h_{x_1}(t)+b_{x_1}(t)^\top\int^\infty_te^{-\lambda(s-t)}\bE_t[\hat{Y}(s)]\rd s+\sum^d_{k=1}\sigma^k_{x_1}(t)^\top\int^\infty_te^{-\lambda(s-t)}\hat{Z}^k(s,t)\rd s\\
	&\hspace{0.3cm}+e^{-\lambda\delta}\bE_t\Bigl[h_{x_2}(t+\delta)+b_{x_2}(t+\delta)^\top\int^\infty_{t+\delta}e^{-\lambda(s-t-\delta)}\bE_{t+\delta}[\hat{Y}(s)]\\
	&\hspace{3cm}+\sum^d_{k=1}\sigma^k_{x_2}(t+\delta)^\top\int^\infty_{t+\delta}e^{-\lambda(s-t-\delta)}\hat{Z}^k(s,t+\delta)\rd s\Bigr],\ t\geq0,
\end{align*}
which admits a unique adapted M-solution $(\hat{Y}(\cdot),\hat{Z}(\cdot,\cdot))\in\cM^{2,-\mu}_\bF(0,\infty;\bR^n\times\bR^{n\times d})$. Defining
\begin{equation}\label{SDDE: ABSDE_BSVIE}
	\hat{\cY}(t):=\int^\infty_te^{-\lambda(s-t)}\bE_t[\hat{Y}(s)]\rd s\ \text{and}\ \hat{\cZ}(t):=\int^\infty_te^{-\lambda(s-t)}\hat{Z}(s,t)\rd s
\end{equation}
for $t\geq0$, the above equation is simplified to
\begin{equation}\label{SDDE: AE_component}
\begin{split}
	\hat{Y}(t)&=h_{x_1}(t)+b_{x_1}(t)^\top\hat{\cY}(t)+\sum^d_{k=1}\sigma^k_{x_1}(t)^\top\hat{\cZ}^k(t)\\
	&+e^{-\lambda\delta}\bE_t\Bigl[h_{x_2}(t+\delta)+b_{x_2}(t+\delta)^\top\hat{\cY}(t+\delta)+\sum^d_{k=1}\sigma^k_{x_2}(t+\delta)^\top\hat{\cZ}^k(t+\delta)\Bigr],\ t\geq0.
\end{split}
\end{equation}
On the other hand, Lemma~3.16 in \cite{Ha21+} yields that $(\hat{\cY}(\cdot),\hat{\cZ}(\cdot))$ is in $L^{2,-\mu}_\bF(0,\infty;\bR^n)\times L^{2,-\mu}_\bF(0,\infty;\bR^{n\times d})$ and that it is the unique adapted solution to the infinite horizon backward stochastic differential equation (BSDE, for short):
\begin{align*}
	\mathrm{d}\hat{\cY}(t)&=-\{\hat{Y}(t)-\lambda\hat{\cY}(t)\}\rd t+\hat{\cZ}(t)\rd W(t),\ t\geq0.
\end{align*}
By inserting the formula \eqref{SDDE: AE_component} into the above equation, we obtain the following \emph{infinite horizon anticipated BSDE} (ABSDE, for short):
\begin{equation}\label{SDDE: AE}
\begin{split}
	\mathrm{d}\hat{\cY}(t)&=-\Bigl\{h_{x_1}(t)+b_{x_1}(t)^\top\hat{\cY}(t)+\sum^d_{k=1}\sigma^k_{x_1}(t)^\top\hat{\cZ}^k(t)\\
	&\hspace{1cm}+e^{-\lambda\delta}\bE_t\Bigl[h_{x_2}(t+\delta)+b_{x_2}(t+\delta)^\top\hat{\cY}(t+\delta)+\sum^d_{k=1}\sigma^k_{x_2}(t+\delta)^\top\hat{\cZ}^k(t+\delta)\Bigr]\Bigr\}\rd t\\
	&\hspace{0.5cm}+\lambda\hat{\cY}(t)\rd t+\hat{\cZ}(t)\rd W(t),\ t\geq0.
\end{split}
\end{equation}
By construction, together with Lemma~3.16 in \cite{Ha21+}, the above infinite horizon ABSDE has a unique adapted solution $(\hat{\cY}(\cdot),\hat{\cZ}(\cdot))\in L^{2,-\mu}_\bF(0,\infty;\bR^n)\times L^{2,-\mu}_\bF(0,\infty;\bR^{n\times d})$ given by \eqref{SDDE: ABSDE_BSVIE}. Furthermore, the optimality condition \eqref{SDVIE: OC} becomes
\begin{equation}\label{SDDE: OC}
	\Bigl\langle h_u(t)+b_u(t)^\top\hat{\cY}(t)+\sum^d_{k=1}\sigma^k_u(t)^\top\hat{\cZ}^k(t),u-\hat{u}(t)\Bigr\rangle\geq0,\ \forall\,u\in U.
\end{equation}

Consequently, we obtain the following (necessary and sufficient) maximum principles for discounted control problems of SDDEs.

%% Theorem

\begin{theo}[The necessary and sufficient maximum principles for SDDEs]\label{SDDE: theo_MP}
Suppose that $b$ and $\sigma$ satisfy \cref{SDDE: assum}, and $h$ satisfies the condition (iii) in \cref{SDVIE: assum}. Let $\mu$ and $\lambda$ satisfy $\mu>\rho_1$ and $\lambda\geq2\mu$. Fix an arbitrary control process $\hat{u}(\cdot)\in\cU_{-\mu}$, and denote by $\hat{X}(\cdot):=X^{\hat{u}}(\cdot)$ the corresponding solution to the controlled SDDE \eqref{1: SDDE_state}. Also, let $(\hat{\cY}(\cdot),\hat{\cZ}(\cdot))\in L^{2,-\mu}_\bF(0,\infty;\bR^n)\times L^{2,-\mu}_\bF(0,\infty;\bR^{n\times d})$ be the unique adapted solution to the infinite horizon ABSDE \eqref{SDDE: AE}. Then the following hold:
\begin{itemize}
\item[(i)]
If $\hat{u}(\cdot)$ is an optimal control for the discounted control problem \eqref{1: SDDE_state}--\eqref{1: cost}, then the optimality condition \eqref{SDDE: OC} holds for a.e.\ $t\geq0$, a.s.
\item[(ii)]
If the map
\begin{equation*}
	\bR^n\times\bR^n\times U\ni(x_1,x_2,u)\mapsto h(t,x_1,x_2,u)+\langle b(t,x_1,x_2,u),\hat{\cY}(t)\rangle+\sum^d_{k=1}\langle\sigma^k(t,x_1,x_2,u),\hat{\cZ}^k(t)\rangle\in\bR
\end{equation*}
is convex for a.e.\ $t\geq0$, a.s., and if the optimality condition \eqref{SDDE: OC} holds for a.e.\ $t\geq0$, a.s., then $\hat{u}(\cdot)$ is an optimal control for the discounted control problem \eqref{1: SDDE_state}--\eqref{1: cost}.
\end{itemize}
\end{theo}

%% Remark

\begin{rem}\label{SDDE: rem_ABSDE}
The above results generalize the frameworks of \cite{ChWu10,OkSuZh11} to the infinite horizon setting. The adjoint equation \eqref{SDDE: AE} is an infinite horizon ABSDE, which is an infinite horizon BSDE whose driver depends on the ``anticipated terms'' $\hat{\cY}(t+\delta)$ and $\hat{\cZ}(t+\delta)$ of the adapted solution. For detailed analysis on (finite horizon) nonlinear ABSDEs, see \cite{PeYa09}.
\end{rem}

%%%%%%%%%%%%%%%%
%%% Section
%%%%%%%%%%%%%%%%

\section{Discounted control problems for fractional SDDEs}\label{sec: FSDDE}

In this section, we consider discounted control problems for fractional SDEs with finite delay, which can be seen as a special case of SDVIEs. First, in Subsection~\ref{subsec: FSDDE_general}, we state the necessary and sufficient maximum principles for general Caputo fractional SDDEs. Then, in Subsection~\ref{subsec: FSDDE_LQ}, we apply the general results to a linear-quadratic (LQ, for short) regulator problem for a Caputo fractional SDDE and give an explicit expression for the optimal control.

%%%%%%%%%%%%%%%%
%%% Subsection
%%%%%%%%%%%%%%%%

\subsection{General results}\label{subsec: FSDDE_general}

Let $b:[0,\infty)\times\bR^n\times\bR^n\times\bR^\ell\to\bR^n$ and $\sigma:[0,\infty)\times\bR^n\times\bR^n\times\bR^\ell\to\bR^{n\times d}$ satisfy \cref{SDDE: assum}, and let a continuous function $(\varphi(t))_{t\in[-\delta,0]}$ be fixed. Let $\delta\geq0$ and $\alpha\in(\frac{1}{2},1]$. Consider the controlled Caputo fractional SDDE~\eqref{1: FSDDE_state}, which we rewrite for readers' convenience:
\begin{equation*}
	\begin{dcases}
	\CD X^u(t)=b(t,X^u(t),X^u(t-\delta),u(t))+\sigma(t,X^u(t),X^u(t-\delta),u(t))\frac{\mathrm{d}W(t)}{\mathrm{d}t},\ t\geq0,\\
	X^u(t)=\varphi(t),\ t\in[-\delta,0],
	\end{dcases}
\end{equation*}
where $\CD$ denotes the Caputo fractional derivative of order $\alpha$ defined by
\begin{equation*}
	\CD f(t):=\frac{1}{\Gamma(1-\alpha)}\frac{\mathrm{d}}{\mathrm{d}t}\int^t_0(t-s)^{-\alpha}\{f(s)-f(0)\}\rd s,\ t\geq0,
\end{equation*}
for suitable function $f:[0,\infty)\to\bR$. Here and elsewhere, $\Gamma(\alpha)=\int^\infty_0e^{-\tau}\tau^{\alpha-1}\rd\tau$ denotes the Gamma function. Following \cite{ZhAgLiPeYoZh17,MoZhLoTeMo20}, we say that $X^u(\cdot)\in L^{2,*}_\bF(0,\infty;\bR^n)$ is a solution of the controlled Caputo fractional SDDE~\eqref{1: FSDDE_state} if it holds that
\begin{equation*}
	\begin{dcases}
	X^u(t)=\varphi(0)+\frac{1}{\Gamma(\alpha)}\int^t_0(t-s)^{\alpha-1}b(s,X^u(s),X^u(s-\delta),u(s))\rd s\\
	\hspace{2cm}+\frac{1}{\Gamma(\alpha)}\int^t_0(t-s)^{\alpha-1}\sigma(s,X^u(s),X^u(s-\delta),u(s))\rd W(s),\ t\geq0,\\
	X^u(t)=\varphi(t),\ t\in[-\delta,0].
	\end{dcases}
\end{equation*}
The above can be seen as an SDVIE with the coefficients
\begin{equation*}
	b(t,s,x_1,x_2,u)=\frac{1}{\Gamma(\alpha)}(t-s)^{\alpha-1}b(s,x_1,x_2,u)\ \text{and}\ \sigma(t,s,x_1,x_2,u)=\frac{1}{\Gamma(\alpha)}(t-s)^{\alpha-1}\sigma(s,x_1,x_2,u).
\end{equation*}
The corresponding Lipschitz coefficients can be taken as
\begin{equation*}
	K_{f,\xi}(\tau)=\frac{L_{f,\xi}}{\Gamma(\alpha)}\tau^{\alpha-1},\ \tau\geq0,\ \xi=x_1,x_2,u,\ f=b,\sigma.
\end{equation*}
Note that the above coefficients are singular in the sense that their Lipschitz coefficients with respect to $x_1$, $x_2$ and $u$ diverge as $s\to t$. We consider the minimization problem for the discounted cost functional \eqref{1: cost}, subject to the controlled Caputo fractional SDDE \eqref{1: FSDDE_state}. We can apply the general results in \cref{sec: SDVIE}. In this case, simple calculations show that $\rho_{b,\sigma;x_1,x_2,u}=\rho_\alpha$, where
\begin{equation}\label{FSDDE: criterion}
	\rho_\alpha:=\inf\left\{\rho>0\relmiddle|(1+e^{-\rho\delta})\Bigl\{(L_{b,x_1}+L_{b,x_2})\rho^{-\alpha}+(L_{\sigma,x_1}+L_{\sigma,x_2})\frac{\sqrt{\Gamma(2\alpha-1)}}{\Gamma(\alpha)}(2\rho)^{-(\alpha-1/2)}\Bigr\}\leq1\right\}.
\end{equation}
Therefore, if $\mu>\rho_\alpha$ and $\lambda\geq2\mu$, then for each control process $u(\cdot)\in\cU_{-\mu}$, there exists a unique solution $X^u(\cdot)\in L^{2,-\mu}_\bF(0,\infty;\bR^n)$ to the controlled Caputo fractional SDDE \eqref{1: FSDDE_state}. Furthermore, the cost functional \eqref{1: cost} is well-defined. Fix an arbitrary control process $\hat{u}(\cdot)$, and denote the corresponding state process $\hat{X}(\cdot):=X^{\hat{u}}(\cdot)$. Applying the observation in \cref{sec: SDVIE}, we see that the corresponding adjoint equation \eqref{SDVIE: AE} becomes the following infinite horizon ABSVIE:
\begin{equation}\label{FSDDE: AE}
\begin{split}
	\hat{Y}(t)&=h_{x_1}(t)+e^{-\lambda\delta}h_{x_2}(t)\\
	&\hspace{1cm}+\frac{1}{\Gamma(\alpha)}\int^\infty_te^{-\lambda(s-t)}(s-t)^{\alpha-1}\Bigl\{b_{x_1}(t)^\top\hat{Y}(s)+e^{-\lambda\delta}\bE_s[b_{x_2}(t+\delta)^\top\hat{Y}(s+\delta)]\\
	&\hspace{2.5cm}+\sum^d_{k=1}\sigma^k_{x_1}(t)^\top\hat{Z}^k(s,t)+\sum^d_{k=1}e^{-\lambda\delta}\bE_s[\sigma^k_{x_2}(t+\delta)^\top\hat{Z}^k(s+\delta,t+\delta)]\Bigr\}\rd s\\
	&\hspace{1cm}-\int^\infty_t\hat{Z}(t,s)\rd W(s),\ t\geq0,
\end{split}
\end{equation}
The above equation admits a unique adapted M-solution $(\hat{Y}(\cdot),\hat{Z}(\cdot,\cdot))\in\cM^{2,-\mu}_\bF(0,\infty;\bR^n\times\bR^{n\times d})$. Furthermore, the optimality condition becomes
\begin{equation}\label{FSDDE: OC}
\begin{split}
	&\Bigl\langle h_u(t)+\frac{1}{\Gamma(\alpha)}b_u(t)^\top\int^\infty_te^{-\lambda(s-t)}(s-t)^{\alpha-1}\bE_t\bigl[\hat{Y}(s)\bigr]\rd s\\
	&\hspace{1cm}+\frac{1}{\Gamma(\alpha)}\sum^d_{k=1}\sigma^k_u(t)^\top\int^\infty_te^{-\lambda(s-t)}(s-t)^{\alpha-1}\hat{Z}^k(s,t)\rd s,u-\hat{u}(t)\Bigr\rangle\geq0,\ \forall\,u\in U.
\end{split}
\end{equation}
By \cref{SDVIE: theo_MP}, if $\hat{u}(\cdot)$ is an optimal control, then the optimality condition \eqref{FSDDE: OC} holds for a.e.\ $t\geq0$, a.s. Conversely, if the map
\begin{equation}\label{FSDDE: convex}
\begin{split}
	&(x_1,x_2,u)\mapsto h(t,x_1,x_2,u)+\frac{1}{\Gamma(\alpha)}\Bigl\langle b(t,x_1,x_2,u),\int^\infty_te^{-\lambda(s-t)}(s-t)^{\alpha-1}\bE_t\bigl[\hat{Y}(s)\bigr]\rd s\Bigr\rangle\\
	&\hspace{4cm}+\frac{1}{\Gamma(\alpha)}\sum^d_{k=1}\Bigl\langle\sigma^k(t,x_1,x_2,u),\int^\infty_te^{-\lambda(s-t)}(s-t)^{\alpha-1}\hat{Z}^k(s,t)\rd s\Bigr\rangle
\end{split}
\end{equation}
is convex for a.e.\ $t\geq0$, a.s., and if \eqref{FSDDE: OC} holds for a.e.\ $t\geq0$, a.s., then the control process $\hat{u}(\cdot)$ is optimal.

%%%%%%%%%%%%%%%%
%%% Subsection
%%%%%%%%%%%%%%%%

\subsection{Infinite horizon LQ regulator problems for fractional SDDEs}\label{subsec: FSDDE_LQ}

As an example, we consider an infinite horizon LQ regulator problem for a Caputo fractional SDDE. For simplicity of notation, in this subsection, we consider the one-dimensional case, that is, $d=n=\ell=1$. Suppose that we are given constants $x_0,b,\sigma\in\bR$, $c\in\bR\setminus\{0\}$, $\gamma>0$, $\alpha\in(\frac{1}{2},1]$, $\delta\geq0$, and $\lambda>0$, and assume that the control set $U$ is the whole real line. Consider the following controlled linear Caputo fractional SDDE:
\begin{equation}\label{LQ: state}
	\begin{dcases}
	\CD X^u(t)=bX^u(t-\delta)+cu(t)+\sigma\frac{\mathrm{d}W(t)}{\mathrm{d}t},\ t\geq0,\\
	X^u(t)=x_0,\ t\in[-\delta,0].
	\end{dcases}
\end{equation}
The cost functional is defined by
\begin{equation}\label{LQ: cost}
	J_\lambda(u(\cdot))=\frac{1}{2}\bE\Bigl[\int^\infty_0e^{-\lambda t}\Bigl\{|X^u(t)|^2+\frac{1}{\gamma}|u(t)|^2\Bigr\}\rd t\Bigr].
\end{equation}
We can apply the general results in Subsection~\ref{subsec: FSDDE_general}. As to the constant \eqref{FSDDE: criterion}, we have $\rho_\alpha=0$ when $b=0$. On the other hand, when $b\neq0$, $\rho_\alpha >0$ is the constant satisfying $|b|(1+e^{-\rho_\alpha\delta})\rho^{-\alpha}_\alpha=1$. Fix $\mu>\rho_\alpha$, and let $\lambda\geq2\mu$. Then for any control process $u(\cdot)\in\cU_{-\mu}=L^{2,-\mu}_\bF(0,\infty;\bR)$, there exists a unique solution $X^u(\cdot)\in L^{2,-\mu}_\bF(0,\infty;\bR)$ to the linear Caputo fractional SDDE \eqref{LQ: state}, and the cost functional \eqref{LQ: cost} is well-defined. Note that, by the definition, $X^u(\cdot)$ solves the following SDVIE:
\begin{equation}\label{LQ: state'}
	\begin{dcases}
	X^u(t)=x_0+\frac{1}{\Gamma(\alpha)}\int^t_0(t-s)^{\alpha-1}\bigl\{bX^u(s-\delta)+cu(s)\bigr\}\rd s+\frac{1}{\Gamma(\alpha)}\int^t_0(t-s)^{\alpha-1}\sigma\rd W(s),\ t\geq0,\\
	X^u(t)=x_0,\ t\in[-\delta,0].
	\end{dcases}
\end{equation}
The LQ regulator problem is to minimize the quadratic cost functional \eqref{LQ: cost} over all control processes $u(\cdot)\in\cU_{-\mu}$, subject to the linear Caputo fractional SDDE \eqref{LQ: state} (or equivalently the linear SDVIE \eqref{LQ: state'}).

We shall show that, when the parameter $\mu>0$ is sufficiently large, there exists a unique optimal control for the LQ regulator problem \eqref{LQ: state}--\eqref{LQ: cost}. Furthermore, we show that the optimal control is given by a \emph{Gaussian state-feedback representation formula}. The precise statement of the main result (\cref{LQ: theo_characterization}) will be given later, since we need some new notations which will appear in the following arguments.

%% Remark

\begin{rem}\label{LQ: rem_other}
\begin{itemize}
\item[(i)]
Our framework includes the cases without delay ($\delta=0$), the deterministic case ($\sigma=0$), and the classical derivative case ($\alpha=1$).
\item[(ii)]
We can easily extend the arguments below to the more general case where the state dynamics~\eqref{LQ: state} and the cost functional~\eqref{LQ: cost} are replaced by
\begin{equation*}
	\begin{dcases}
	\CD X^u(t)=b_1X^u(t)+b_2X^u(t-\delta)+cu(t)+\sigma\frac{\mathrm{d}W(t)}{\mathrm{d}t},\ t\geq0,\\
	X^u(t)=x_0,\ t\in[-\delta,0],
	\end{dcases}
\end{equation*}
and
\begin{equation*}
	J_\lambda(u(\cdot))=\bE\Bigl[\int^\infty_0e^{-\lambda t}\Bigl\{h_1|X^u(t)|^2+h_2|X^u(t-\delta)|^2+\frac{1}{\gamma}|u(t)|^2\Bigr\}\rd t\Bigr],
\end{equation*}
respectively, for some constants $b_1,b_2\in\bR$ and $h_1,h_2>0$. However, in the following, we consider \eqref{LQ: state} and \eqref{LQ: cost} for notational simplicity.
\end{itemize}
\end{rem}

First, we prove the uniqueness of the optimal control.

%% Lemma

\begin{lemm}\label{LQ: lemm_unique}
Let $\mu>\rho_\alpha$ and $\lambda\geq2\mu$. Then the LQ regulator problem \eqref{LQ: state}--\eqref{LQ: cost} has at most one optimal control in $\cU_{-\mu}$.
\end{lemm}

%% Proof

\begin{proof}
We show that the cost functional $u(\cdot)\mapsto J_\lambda(u(\cdot))$ is strictly convex on $\cU_{-\mu}$. Let $u_0(\cdot),u_1(\cdot)\in\cU_{-\mu}$ be two control processes such that $u_0(t)\neq u_1(t)$ with positive measure with respect to $\mathrm{Leb}\otimes\bP$, where $\mathrm{Leb}$ denotes the Lebesgue measure on $[0,\infty)$. Denote the corresponding state processes by $X_0(\cdot):=X^{u_0}(\cdot)$ and $X_1(\cdot):=X^{u_1}(\cdot)$, respectively. For each $\theta\in(0,1)$, define $u_\theta(\cdot):=(1-\theta)u_0(\cdot)+\theta u_1(\cdot)\in\cU_{-\mu}$ and $X_\theta(\cdot):=X^{u_\theta}(\cdot)$. By the uniqueness of the solution of the controlled linear Caputo fractional SDDE \eqref{LQ: state}, we see that $X_\theta(\cdot)=(1-\theta)X_0(\cdot)+\theta X_1(\cdot)$. Hence, by the quadratic structure of the cost functional, we can easily show that $J_\lambda(u_\theta(\cdot))<(1-\theta)J_\lambda(u_0(\cdot))+\theta J_\lambda(u_1(\cdot))$. This implies that $u(\cdot)\mapsto J_\lambda(u(\cdot))$ is strictly convex on $\cU_{-\mu}$, and thus the optimal control is, if it exists, unique.
\end{proof}

Next, we give a characterization of the optimal control in several steps. The following lemma is the first step, which is based on the necessary and sufficient maximum principles.

%% Lemma

\begin{lemm}\label{LQ: lemm_OC1}
Let $\mu>\rho_\alpha$ and $\lambda\geq2\mu$. A control process $\hat{u}(\cdot)\in\cU_{-\mu}$ is optimal if and only if the following holds:
\begin{equation}\label{LQ: OC1}
	\hat{u}(t)+\frac{c\gamma}{\Gamma(\alpha)}\int^\infty_te^{-\lambda(s-t)}(s-t)^{\alpha-1}\bE_t[\hat{X}(s)]\rd s-\frac{be^{-\lambda\delta}}{\Gamma(\alpha)}\int^\infty_te^{-\lambda(s-t)}(s-t)^{\alpha-1}\bE_t[\hat{u}(s+\delta)]\rd s=0
\end{equation}
for a.e.\ $t\geq0$, a.s., where $\hat{X}(\cdot):=X^{\hat{u}}(\cdot)$ denotes the state process corresponding to $\hat{u}(\cdot)$.
\end{lemm}

%% Proof

\begin{proof}
For each $\hat{u}(\cdot)\in\cU_{-\mu}$, the corresponding adjoint equation \eqref{FSDDE: AE} becomes
\begin{equation*}
	\hat{Y}(t)=\hat{X}(t)+\frac{be^{-\lambda\delta}}{\Gamma(\alpha)}\int^\infty_te^{-\lambda(s-t)}(s-t)^{\alpha-1}\bE_s[\hat{Y}(s+\delta)]\rd s-\int^\infty_t\hat{Z}(t,s)\rd W(s),\ t\geq0,
\end{equation*}
which admits a unique adapted M-solution $(\hat{Y}(\cdot),\hat{Z}(\cdot,\cdot))\in\cM^{2,-\mu}_\bF(0,\infty;\bR\times\bR)$. By the change of variable formula and the tower property of conditional expectations, the above equation can be rewritten as
\begin{equation}\label{LQ: AE}
	\hat{Y}(t)=\hat{X}(t)+be^{-\lambda\delta}\bE_t\Bigl[\frac{1}{\Gamma(\alpha)}\int^\infty_{t+\delta}e^{-\lambda(s-t-\delta)}(s-t-\delta)^{\alpha-1}\bE_{t+\delta}[\hat{Y}(s)]\rd s\Bigr],\ t\geq0.
\end{equation}
Also, noting that $U=\bR$, the optimality condition \eqref{FSDDE: OC} becomes
\begin{equation}\label{LQ: OC0}
	\frac{1}{\gamma}\hat{u}(t)+\frac{c}{\Gamma(\alpha)}\int^\infty_te^{-\lambda(s-t)}(s-t)^{\alpha-1}\bE_t[\hat{Y}(s)]\rd s=0.
\end{equation}
Clearly, the map \eqref{FSDDE: convex} is convex. Thus, by the necessary and sufficient maximum principles (\cref{SDVIE: theo_MP}), $\hat{u}(\cdot)$ is optimal if and only if \eqref{LQ: OC0} holds for a.e.\ $t\geq0$, a.s. Now we show that \eqref{LQ: OC0} is equivalent to \eqref{LQ: OC1}.

First, assume that \eqref{LQ: OC0} holds for a.e.\ $t\geq0$, a.s. Then by inserting \eqref{LQ: OC0} with $t$ replaced by $t+\delta$ into the adjoint equation \eqref{LQ: AE}, we have
\begin{equation*}
	\hat{Y}(t)=\hat{X}(t)-\frac{be^{-\lambda\delta}}{c\gamma}\bE_t[\hat{u}(t+\delta)],\ t\geq0.
\end{equation*}
Successively, by inserting this formula into \eqref{LQ: OC0}, we get \eqref{LQ: OC1}.

Conversely, assume that \eqref{LQ: OC1} holds for a.e.\ $t\geq0$, a.s. Define
\begin{equation*}
	Y(t):=\hat{X}(t)-\frac{be^{-\lambda\delta}}{c\gamma}\bE_t[\hat{u}(t+\delta)],\ t\geq0.
\end{equation*}
Then $Y(\cdot)\in L^{2,-\mu}_\bF(0,\infty;\bR)$, and \eqref{LQ: OC1} is written as
\begin{equation*}
	\hat{u}(t)+\frac{c\gamma}{\Gamma(\alpha)}\int^\infty_te^{-\lambda(s-t)}(s-t)^{\alpha-1}\bE_t[Y(s)]\rd s=0.
\end{equation*}
By inserting this formula with $t$ replaced by $t+\delta$ into the definition of $Y(\cdot)$, we see that
\begin{equation*}
	Y(t)=\hat{X}(t)+be^{-\lambda\delta}\bE_t\Bigl[\frac{1}{\Gamma(\alpha)}\int^\infty_{t+\delta}e^{-\lambda(s-t-\delta)}(s-t-\delta)^{\alpha-1}\bE_{t+\delta}[Y(s)]\rd s\Bigr],\ t\geq0.
\end{equation*}
Therefore, by the uniqueness of the solution to the equation \eqref{LQ: AE} in $L^{2,-\mu}_\bF(0,\infty;\bR)$, we have $Y(t)=\hat{Y}(t)$ for a.e.\ $t\geq0$, a.s., and thus \eqref{LQ: OC0} holds. This completes the proof.
\end{proof}

Now we transform the state-control pair. For each control process $u(\cdot)$, define $\cT[u](\cdot)\in\cU_{-\mu}$ by $\cT[u](t):=\frac{b}{c}X^u(t-\delta)+u(t)$ for $t\geq0$. Noting the uniqueness of the solution to the controlled fractional SDDE~\eqref{LQ: state}, it can be easily shown that the map $\cT$ is bijective on $\cU_{-\mu}$, and the inverse map $\cT^{-1}$ is given by $\cT^{-1}[v](t)=-\frac{b}{c}\cX^v(t-\delta)+v(t)$, $t\geq0$, for each $v(\cdot)\in\cU_{-\mu}$, where
\begin{equation*}
	\begin{dcases}
	\cX^v(t):=x_0+\frac{c}{\Gamma(\alpha)}\int^t_0(t-s)^{\alpha-1}v(s)\rd s+\frac{\sigma}{\Gamma(\alpha)}\int^t_0(t-s)^{\alpha-1}\rd W(s),\ t\geq0,\\
	\cX^v(t):=x_0,\ t\in[-\delta,0].
	\end{dcases}
\end{equation*}
Furthermore, it holds that $\cX^v(\cdot)=X^{\cT^{-1}[v]}(\cdot)$ for each $v(\cdot)\in\cU_{-\mu}$. The following lemma is the second step to characterize the optimal control.

%% Lemma

\begin{lemm}\label{LQ: lemm_OC2}
Let $\mu>\rho_\alpha$ and $\lambda\geq2\mu$. A control process $\hat{u}(\cdot)\in\cU_{-\mu}$ is optimal if and only if the control process $\hat{v}(\cdot):=\cT[\hat{u}](\cdot)\in\cU_{-\mu}$ satisfies the following \emph{stochastic Fredholm integral equation}:
\begin{equation}\label{LQ: SFIE}
	\hat{v}(t)+\int^t_0g_\lambda(t-s)\hat{v}(s)\rd s+\int^\infty_te^{-\lambda(s-t)}g_\lambda(s-t)\bE_t[\hat{v}(s)]\rd s+\frac{\sigma}{c}\int^t_0g_\lambda(t-s)\rd W(s)+K_\lambda x_0=0,\ t\geq0,
\end{equation}
where $K_\lambda\in\bR$ and $g_\lambda:[0,\infty)\to\bR$ are defined by
\begin{equation}\label{LQ: K_lambda}
	K_\lambda:=\frac{\bigl(c^2\gamma+b^2e^{-\lambda\delta}\bigr)\lambda^{-\alpha}-b}{c}
\end{equation}
and
\begin{equation}\label{LQ: g_lambda}
	g_\lambda(\tau):=\frac{c^2\gamma+b^2e^{-\lambda\delta}}{\Gamma(\alpha)^2}\int^\infty_0e^{-\lambda\theta}\theta^{\alpha-1}(\theta+\tau)^{\alpha-1}\rd\theta-\frac{b}{\Gamma(\alpha)}(\tau-\delta)^{\alpha-1}_+,\ \tau\geq0,
\end{equation}
respectively. Here, we define $x^{\alpha-1}_+:=x^{\alpha-1}\1_{(0,\infty)}(x)$ for any $x\in\bR$.
\end{lemm}

%% Proof

\begin{proof}
Let $\hat{u}(\cdot)\in\cU_{-\mu}$ be fixed, and define $\hat{v}(\cdot):=\cT[\hat{u}](\cdot)\in\cU_{-\mu}$. By the definition, it holds that $\hat{u}(t)=\cT^{-1}[\hat{v}](t)=-\frac{b}{c}\hat{\cX}(t-\delta)+\hat{v}(t)$, $t\geq0$, where $\hat{\cX}(\cdot):=\cX^{\hat{v}}(\cdot)\in L^{2,-\mu}_\bF(0,\infty;\bR)$. Furthermore, we have $\hat{X}(\cdot):=X^{\hat{u}}(\cdot)=\hat{\cX}(\cdot)$. Therefore, the left-hand side of the optimality condition \eqref{LQ: OC1} is rewritten as follows:
\begin{equation}\label{LQ: OC1_LHS}
\begin{split}
	&\hat{u}(t)+\frac{c\gamma}{\Gamma(\alpha)}\int^\infty_te^{-\lambda(s-t)}(s-t)^{\alpha-1}\bE_t[\hat{X}(s)]\rd s-\frac{be^{-\lambda\delta}}{\Gamma(\alpha)}\int^\infty_te^{-\lambda(s-t)}(s-t)^{\alpha-1}\bE_t[\hat{u}(s+\delta)]\rd s\\
	&=-\frac{b}{c}\hat{\cX}(t-\delta)+\hat{v}(t)+\frac{c^2\gamma+b^2e^{-\lambda\delta}}{c\Gamma(\alpha)}\int^\infty_te^{-\lambda(s-t)}(s-t)^{\alpha-1}\bE_t[\hat{\cX}(s)]\rd s\\
	&\hspace{1cm}-\frac{be^{-\lambda\delta}}{\Gamma(\alpha)}\int^\infty_te^{-\lambda(s-t)}(s-t)^{\alpha-1}\bE_t[\hat{v}(s+\delta)]\rd s.
\end{split}
\end{equation}
Observe that
\begin{equation*}
	\hat{\cX}(t-\delta)=x_0+\frac{c}{\Gamma(\alpha)}\int^t_0(t-s-\delta)^{\alpha-1}_+\hat{v}(s)\rd s+\frac{\sigma}{\Gamma(\alpha)}\int^t_0(t-s-\delta)^{\alpha-1}_+\rd W(s),\ t\geq0,
\end{equation*}
where $x^{\alpha-1}_+:=x^{\alpha-1}\1_{(0,\infty)}(x)$ for any $x\in\bR$. Also, by the change of variable formula, we have
\begin{align*}
	\int^\infty_te^{-\lambda(s-t)}(s-t)^{\alpha-1}\bE_t[\hat{v}(s+\delta)]\rd s&=\int^\infty_{t+\delta}e^{-\lambda(s-t-\delta)}(s-t-\delta)^{\alpha-1}\bE_t[\hat{v}(s)]\rd s\\
	&=e^{\lambda\delta}\int^\infty_te^{-\lambda(s-t)}(s-t-\delta)^{\alpha-1}_+\bE_t[\hat{v}(s)]\rd s,\ t\geq0.
\end{align*}
By using (stochastic) Fubini's theorem, the integral term containing the conditional expectation of $\hat{\cX}(\cdot)$ can be calculated as follows:
\begin{align*}
	&\int^\infty_te^{-\lambda(s-t)}(s-t)^{\alpha-1}\bE_t[\hat{\cX}(s)]\rd s\\
	&=\int^\infty_te^{-\lambda(s-t)}(s-t)^{\alpha-1}\Bigl\{x_0+\frac{c}{\Gamma(\alpha)}\int^t_0(s-\theta)^{\alpha-1}\hat{v}(\theta)\rd\theta+\frac{c}{\Gamma(\alpha)}\int^s_t(s-\theta)^{\alpha-1}\bE_t[\hat{v}(\theta)]\rd\theta\\
	&\hspace{5cm}+\frac{\sigma}{\Gamma(\alpha)}\int^t_0(s-\theta)^{\alpha-1}\rd W(\theta)\Bigr\}\rd s\\
	&=\lambda^{-\alpha}\Gamma(\alpha)x_0+\frac{c}{\Gamma(\alpha)}\int^t_0\int^\infty_te^{-\lambda(\theta-t)}(\theta-t)^{\alpha-1}(\theta-s)^{\alpha-1}\rd\theta\,\hat{v}(s)\rd s\\
	&\hspace{2cm}+\frac{c}{\Gamma(\alpha)}\int^\infty_t\int^\infty_se^{-\lambda(\theta-t)}(\theta-t)^{\alpha-1}(\theta-s)^{\alpha-1}\rd\theta\,\bE_t[\hat{v}(s)]\rd s\\
	&\hspace{2cm}+\frac{\sigma}{\Gamma(\alpha)}\int^t_0\int^\infty_te^{-\lambda(\theta-t)}(\theta-t)^{\alpha-1}(\theta-s)^{\alpha-1}\rd\theta\rd W(s).
\end{align*}
Note that, for $(t,s)\in\Delta^\comp[0,\infty)$,
\begin{equation*}
	\frac{1}{\Gamma(\alpha)}\int^\infty_te^{-\lambda(\theta-t)}(\theta-t)^{\alpha-1}(\theta-s)^{\alpha-1}\rd\theta=\frac{1}{\Gamma(\alpha)}\int^\infty_0e^{-\lambda\theta}\theta^{\alpha-1}(\theta+t-s)^{\alpha-1}\rd\theta=f_\lambda(t-s),
\end{equation*}
where
\begin{equation*}
	f_\lambda(\tau):=\frac{1}{\Gamma(\alpha)}\int^\infty_0e^{-\lambda\theta}\theta^{\alpha-1}(\theta+\tau)^{\alpha-1}\rd\theta,\ \tau\geq0.
\end{equation*}
Also, for $(t,s)\in\Delta[0,\infty)$,
\begin{equation*}
	\frac{1}{\Gamma(\alpha)}\int^\infty_se^{-\lambda(\theta-t)}(\theta-t)^{\alpha-1}(\theta-s)^{\alpha-1}\rd\theta=\frac{e^{-\lambda(s-t)}}{\Gamma(\alpha)}\int^\infty_0e^{-\lambda\theta}\theta^{\alpha-1}(\theta+s-t)^{\alpha-1}\rd\theta=e^{-\lambda(s-t)}f_\lambda(s-t).
\end{equation*}
Therefore, we obtain
\begin{align*}
	&\int^\infty_te^{-\lambda(s-t)}(s-t)^{\alpha-1}\bE_t[\hat{\cX}(s)]\rd s\\
	&=\lambda^{-\alpha}\Gamma(\alpha)x_0+c\int^t_0f_\lambda(t-s)\hat{v}(s)\rd s+c\int^\infty_te^{-\lambda(s-t)}f_\lambda(s-t)\bE_t[\hat{v}(s)]\rd s+\sigma\int^t_0f_\lambda(t-s)\rd W(s),\ t\geq0.
\end{align*}
The above calculations yield that the right-hand side of the equality \eqref{LQ: OC1_LHS} is equal to
\begin{align*}
	&-\frac{b}{c}\Bigl\{x_0+\frac{c}{\Gamma(\alpha)}\int^t_0(t-s-\delta)^{\alpha-1}_+\hat{v}(s)\rd s+\frac{\sigma}{\Gamma(\alpha)}\int^t_0(t-s-\delta)^{\alpha-1}_+\rd W(s)\Bigr\}+\hat{v}(t)\\
	&+\frac{c^2\gamma+b^2e^{-\lambda\delta}}{c\Gamma(\alpha)}\Big\{\lambda^{-\alpha}\Gamma(\alpha)x_0+c\int^t_0f_\lambda(t-s)\hat{v}(s)\rd s+c\int^\infty_te^{-\lambda(s-t)}f_\lambda(s-t)\bE_t[\hat{v}(s)]\rd s\\
	&\hspace{3cm}+\sigma\int^t_0f_\lambda(t-s)\rd W(s)\Bigr\}\\
	&-\frac{be^{-\lambda\delta}}{\Gamma(\alpha)}e^{\lambda\delta}\int^\infty_te^{-\lambda(s-t)}(s-t-\delta)^{\alpha-1}_+\bE_t[\hat{v}(s)]\rd s\\
	&=\hat{v}(t)+\frac{1}{\Gamma(\alpha)}\int^t_0\bigl\{\bigl(c^2\gamma+b^2e^{-\lambda\delta}\bigr)f_\lambda(t-s)-b(t-s-\delta)^{\alpha-1}_+\bigr\}\hat{v}(s)\rd s\\
	&\hspace{1cm}+\frac{1}{\Gamma(\alpha)}\int^\infty_te^{-\lambda(s-t)}\bigl\{\bigl(c^2\gamma+b^2e^{-\lambda\delta}\bigr)f_\lambda(s-t)-b(s-t-\delta)^{\alpha-1}_+\bigr\}\bE_t[\hat{v}(s)]\rd s\\
	&\hspace{1cm}+\frac{\sigma}{c\Gamma(\alpha)}\int^t_0\bigl\{\bigl(c^2\gamma+b^2e^{-\lambda\delta}\bigr)f_\lambda(t-s)-b(t-s-\delta)^{\alpha-1}_+\bigr\}\rd W(s)\\
	&\hspace{1cm}+\frac{\bigl(c^2\gamma+b^2e^{-\lambda\delta}\bigr)\lambda^{-\alpha}-b}{c}x_0\\
	&=\hat{v}(t)+\int^t_0g_\lambda(t-s)\hat{v}(s)\rd s+\int^\infty_te^{-\lambda(s-t)}g_\lambda(s-t)\bE_t[\hat{v}(s)]\rd s+\frac{\sigma}{c}\int^t_0g_\lambda(t-s)\rd W(s)+K_\lambda x_0,
\end{align*}
where $K_\lambda\in\bR$ and $g_\lambda:[0,\infty)\to\bR$ are defined by \eqref{LQ: K_lambda} and \eqref{LQ: g_lambda}, respectively. Therefore, by \cref{LQ: lemm_OC1}, we get the assertions.
\end{proof}

%% Remark

\begin{rem}\label{LQ: rem_g_lambda}
Noting that $\alpha\in(\frac{1}{2},1]$, simple calculations show that
\begin{equation}\label{LQ: g_lambda_estimate}
	|g_\lambda(\tau)|\leq\frac{c^2\gamma+b^2e^{-\lambda\delta}}{\Gamma(\alpha)}\lambda^{-\alpha}\tau^{\alpha-1}+\frac{|b|}{\Gamma(\alpha)}(\tau-\delta)^{\alpha-1}_+,\ \forall\,\tau>0,
\end{equation}
and that $g_\lambda\in L^{2,\beta}(0,\infty;\bR)$ for any $\beta<0$.
\end{rem}

By \cref{LQ: lemm_OC2}, the existence of an optimal control of the LQ regulator problem \eqref{LQ: state}--\eqref{LQ: cost} is characterized by the solvability of the (linear) stochastic Fredholm integral equation \eqref{LQ: SFIE}. In the following, we show that the stochastic Fredholm integral equation \eqref{LQ: SFIE} admits a unique solution $\hat{v}(\cdot)\in\cU_{-\mu}=L^{2,-\mu}_\bF(0,\infty;\bR)$ when the parameter $\mu>0$ is sufficiently large. A main idea is to divide the equation \eqref{LQ: SFIE} into two deterministic Fredholm integral equations which correspond to the ``expectation part'' and the ``martingale integrand part'' of $\hat{v}(\cdot)$, respectively, by means of the martingale representation theorem.

To do so, we need the following abstract lemma. We state the result in a multi-dimensional setting (that is, the Brownian motion $W(\cdot)$ is $d$-dimensional with $d\in\bN$).

%% Lemma

\begin{lemm}\label{LQ: lemm_MR}
For any $v(\cdot)\in L^{2,\eta}_\bF(0,\infty;\bR^m)$ with $\eta\in\bR$ and $m\in\bN$, there exists a unique pair $(\varphi,\psi)$ of a deterministic function $\varphi:[0,\infty)\to\bR^m$ and a random field $\psi:\Omega\times[0,\infty)^2\to\bR^{m\times d}$ satisfying the following conditions:
\begin{itemize}
\item[(i)]
$\varphi\in L^{2,\eta}(0,\infty;\bR^m)$;
\item[(ii)]
$\psi(\cdot,\cdot)$ is measurable, $\psi(t,\cdot)$ is adapted for a.e.\ $t\geq0$, and $\bE\bigl[\int^\infty_0\int^\infty_0e^{2\eta(t+s)}|\psi(t,s)|^2\rd s\rd t\bigr]<\infty$;
\item[(iii)]
$v(t)=\varphi(t)+\int^t_0\psi(t-\theta,\theta)\rd W(\theta)$ for a.e.\ $t\geq0$, a.s.
\end{itemize}
\end{lemm}

%% Proof

\begin{proof}
The uniqueness is clear. We prove the existence. Assume that the map $t\mapsto v(t)\in L^2_{\cF_\infty}(\Omega;\bR^m)$ is continuous. By the martingale representation theorem, for any $t\geq0$, there exists a unique process $z(t,\cdot)\in L^2_\bF(0,\infty;\bR^{m\times d})$ with $z(t,s)=0$ for $s>t$ such that $v(t)=\bE[v(t)]+\int^t_0z(t,s)\rd W(s)$ a.s. It is clear that the map $t\mapsto\bE[v(t)]$ is continuous. Furthermore, since $\bE\bigl[\int^\infty_0|z(t',s)-z(t,s)|^2\rd s\bigr]\leq\bE\bigl[|v(t')-v(t)|^2\bigr]\to0$ as $t'\to t$ for any $t\geq0$, we see that the map $t\mapsto z(t,\cdot)\in L^2_\bF(0,\infty;\bR^{m\times d})$ is also continuous. Thus, there exists a jointly measurable version $\tilde{z}:\Omega\times[0,\infty)^2\to\bR^{m\times d}$ of $z(\cdot,\cdot)$. Define $\varphi(t):=\bE[v(t)]$ and $\psi(t,s):=\tilde{z}(t+s,s)$ for $t,s\geq0$. Then $\varphi:[0,\infty)\to\bR^m$ and $\psi:\Omega\times[0,\infty)^2\to\bR^{m\times d}$ are jointly measurable, the process $\psi(t,\cdot)$ is adapted for each $t\geq0$, and it holds that $v(t)=\varphi(t)+\int^t_0\psi(t-\theta,\theta)\rd W(\theta)$ a.s.\ for any $t\geq0$. Moreover, by using Fubini's theorem and the change of variable formula,
\begin{align*}
	\bE\Bigl[\int^\infty_0e^{2\eta t}|v(t)|^2\rd t\Bigr]&=\int^\infty_0e^{2\eta t}|\varphi(t)|^2\rd t+\bE\Bigl[\int^\infty_0e^{2\eta t}\int^t_0|\psi(t-\theta,\theta)|^2\rd \theta\rd t\Bigr]\\
	&=\int^\infty_0e^{2\eta t}|\varphi(t)|^2\rd t+\bE\Bigl[\int^\infty_0e^{2\eta \theta}\int^\infty_\theta e^{2\eta(t-\theta)}|\psi(t-\theta,\theta)|^2\rd t\rd \theta\Bigr]\\
	&=\int^\infty_0e^{2\eta t}|\varphi(t)|^2\rd t+\bE\Bigl[\int^\infty_0e^{2\eta s}\int^\infty_0e^{2\eta t}|\psi(t,s)|^2\rd t\rd s\Bigr]\\
	&=\int^\infty_0e^{2\eta t}|\varphi(t)|^2\rd t+\bE\Bigl[\int^\infty_0\int^\infty_0e^{2\eta(t+s)}|\psi(t,s)|^2\rd s\rd t\Bigr],
\end{align*}
and hence $\varphi\in L^{2,\eta}(0,\infty;\bR^m)$ and $\bE\bigl[\int^\infty_0\int^\infty_0e^{2\eta(t+s)}|\psi(t,s)|^2\rd s\rd t\bigr]<\infty$. This proves the lemma when $t\mapsto v(t)\in L^2_{\cF_\infty}(\Omega,\bR^m)$ is continuous. The general case can be proved by a standard approximation technique (see Lemma~3.5 in \cite{Ha21+}), and thus we omit the details.
\end{proof}

The above lemma is a simple application of the martingale representation theorem and a reparametrization of the time-parameters. Note that, in the martingale integrand for $v(t)$, we use the parametrization of the form $\psi(t-\theta,\theta)\rd W(\theta)$, not of the form $z(t,s)\rd W(s)$. Surprisingly, it turns out that this simple reparametrization is appropriate for the study of the stochastic Fredholm integral equation \eqref{LQ: SFIE}, as shown in the following lemma.

%% Lemma

\begin{lemm}\label{LQ: lemm_FIE}
Let $\mu>0$ and $\lambda\geq2\mu$. Fix an arbitrary $\hat{v}(\cdot)\in \cU_{-\mu}=L^{2,-\mu}_\bF(0,\infty;\bR)$, and denote the corresponding pair appearing in \cref{LQ: lemm_MR} by $(\hat{\varphi},\hat{\psi})$. Then the following two conditions are equivalent:
\begin{itemize}
\item[(i)]
$\hat{v}(\cdot)$ satisfies the stochastic Fredholm integral equation \eqref{LQ: SFIE}.
\item[(ii)]
The function $\hat{\varphi}\in L^{2,-\mu}(0,\infty;\bR)$ satisfies the deterministic Fredholm integral equation
\begin{equation}\label{LQ: DFIE_phi}
	\hat{\varphi}(t)+\int^\infty_0k_\lambda(t,s)\hat{\varphi}(s)\rd s+K_\lambda x_0=0,\ t\geq0.
\end{equation}
Furthermore, for a.e.\ $\theta\geq0$ and a.e.\ $\omega\in\Omega$, the function $\hat{\psi}\in L^{2,-\mu}(0,\infty;\bR)$ defined by $\hat{\psi}(t):=\hat{\psi}(t,\theta)(\omega)$, $t\geq0$, satisfies the deterministic Fredholm integral equation
\begin{equation}\label{LQ: DFIE_psi}
	\hat{\psi}(t)+\int^\infty_0k_\lambda(t,s)\hat{\psi}(s)\rd s+\frac{\sigma}{c}g_\lambda(t)=0,\ t\geq0.
\end{equation}
Here, $k_\lambda:[0,\infty)^2\to\bR$ is defined by
\begin{equation}\label{LQ: k_lambda}
	k_\lambda(t,s):=
	\begin{dcases}
	g_\lambda(t-s)\ &\text{for}\ (t,s)\in\Delta^\comp[0,\infty),\\
	e^{-\lambda(s-t)}g_\lambda(s-t)\ &\text{for}\ (t,s)\in\Delta[0,\infty).
	\end{dcases}
\end{equation}
\end{itemize}
\end{lemm}

%% Proof

\begin{proof}
By using the representation formula $\hat{v}(t)=\hat{\varphi}(t)+\int^t_0\hat{\psi}(t-\theta,\theta)\rd W(\theta)$ and the stochastic Fubini's theorem, we have
\begin{align*}
	&\hat{v}(t)+\int^t_0g_\lambda(t-s)\hat{v}(s)\rd s+\int^\infty_te^{-\lambda(s-t)}g_\lambda(s-t)\bE_t[\hat{v}(s)]\rd s+\frac{\sigma}{c}\int^t_0g_\lambda(t-s)\rd W(s)+K_\lambda x_0\\
	&=\hat{\varphi}(t)+\int^t_0\hat{\psi}(t-\theta,\theta)\rd W(\theta)+\int^t_0g_\lambda(t-s)\Bigl\{\hat{\varphi}(s)+\int^s_0\hat{\psi}(s-\theta,\theta)\rd W(\theta)\Bigr\}\rd s\\
	&\hspace{1cm}+\int^\infty_te^{-\lambda(s-t)}g_\lambda(s-t)\Bigl\{\hat{\varphi}(s)+\int^t_0\hat{\psi}(s-\theta,\theta)\rd W(\theta)\Bigr\}\rd s+\frac{\sigma}{c}\int^t_0g_\lambda(t-\theta)\rd W(\theta)+K_\lambda x_0\\
	&=\hat{\varphi}(t)+\int^t_0g_\lambda(t-s)\hat{\varphi}(s)\rd s+\int^\infty_te^{-\lambda(s-t)}g_\lambda(s-t)\hat{\varphi}(s)\rd s+K_\lambda x_0\\
	&\hspace{1cm}+\int^t_0\Bigl\{\hat{\psi}(t-\theta,\theta)+\int^t_\theta g_\lambda(t-s)\hat{\psi}(s-\theta,\theta)\rd s\\
	&\hspace{3cm}+\int^\infty_te^{-\lambda(s-t)}g_\lambda(s-t)\hat{\psi}(s-\theta,\theta)\rd s+\frac{\sigma}{c}g_\lambda(t-\theta)\Bigr\}\rd W(\theta)\\
	&=\hat{\varphi}(t)+\int^\infty_0k_\lambda(t,s)\hat{\varphi}(s)\rd s+K_\lambda x_0\\
	&\hspace{1cm}+\int^t_0\Bigl\{\hat{\psi}(t-\theta,\theta)+\int^\infty_0 k_\lambda(t-\theta,s)\hat{\psi}(s,\theta)\rd s+\frac{\sigma}{c}g_\lambda(t-\theta)\Bigr\}\rd W(\theta)
\end{align*}
for a.e.\ $t\geq0$, a.s. Since $\hat{\varphi}$ is deterministic, for each $t\geq0$, the above is equal to zero a.s.\ if and only if
\begin{equation*}
	\hat{\varphi}(t)+\int^\infty_0k_\lambda(t,s)\hat{\varphi}(s)\rd s+K_\lambda x_0=0
\end{equation*}
and
\begin{equation*}
	\hat{\psi}(t-\theta,\theta)+\int^\infty_0 k_\lambda(t-\theta,s)\hat{\psi}(s,\theta)\rd s+\frac{\sigma}{c}g_\lambda(t-\theta)=0
\end{equation*}
for a.e.\ $\theta\in[0,t]$, a.s. Clearly, the last equality holds for a.e.\ $t\geq0$ and a.e.\ $\theta\in[0,t]$ if and only if
\begin{equation*}
	\hat{\psi}(t,\theta)+\int^\infty_0 k_\lambda(t,s)\hat{\psi}(s,\theta)\rd s+\frac{\sigma}{c}g_\lambda(t)=0
\end{equation*}
for a.e.\ $t\geq0$ and a.e.\ $\theta\geq0$. This completes the proof.
\end{proof}

\begin{rem}\label{5-2: rem: FIE}
It is remarkable that if $\hat{v}(t)=\hat{\varphi}(t)+\int^t_0\hat{\psi}(t-\theta,\theta)\rd W(\theta)$ satisfies the stochastic Fredholm integral equation \eqref{LQ: SFIE}, then the function $t\mapsto\hat{\psi}(t,s)(\omega)$ solves the common deterministic Fredholm integral equation \eqref{LQ: DFIE_psi} for a.e.\ $s\geq0$, a.s. Therefore, if \eqref{LQ: DFIE_psi} has at most one solution, then $\hat{\psi}(t)=\hat{\psi}(t,s)(\omega)$ is deterministic and independent of the second time-parameter $s\geq0$, and thus $\hat{v}(\cdot)$ must be a Gaussian process of the form of a stochastic convolution $\hat{v}(t)=\hat{\varphi}(t)+\int^t_0\hat{\psi}(t-\theta)\rd W(\theta)$.
\end{rem}

From the above lemma, we see that the linear deterministic Fredholm integral equations of the form
\begin{equation}\label{LQ: DFIE_abstract}
	x(t)+\int^\infty_0k_\lambda(t,s)x(s)\rd s=a(t),\ t\geq0,
\end{equation}
with $a(t)=-K_\lambda x_0$ and $a(t)=-\frac{\sigma}{c}g_\lambda(t)$ characterize the optimal control. This type of equations can be treated by means of the resolvent of the kernel $k_\lambda$. For the general theory on deterministic Fredholm integral equations, we refer to Chapter~9 of the textbook \cite{GrLoSt90}. Specifically, as in Definitions~9.2.1 and 9.2.2 in \cite{GrLoSt90}, for each $p,q\in[1,\infty]$ satisfying $\frac{1}{p}+\frac{1}{q}=1$, we say that $k:[0,\infty)^2\to\bR$ is a \emph{Fredholm kernel} of type $L^p$ on $[0,\infty)$ if it is measurable and $\res k\res_{L^p(0,\infty;\bR)}<\infty$, where
\begin{equation*}
	\res k\res_{L^p(0,\infty;\bR)}:=\sup_{\substack{\|g\|_{L^q(0,\infty;\bR)\leq1}\\\|f\|_{L^p(0,\infty;\bR)\leq1}}}\int^\infty_0\int^\infty_0|g(t)k(t,s)f(s)|\rd s\rd t,
\end{equation*}
and the supremum is taken over functions $g$ and $f$ in $L^q(0,\infty;\bR)$ and $L^p(0,\infty;\bR)$, respectively. Furthermore, for each $\eta\in\bR$, we say that $k$ is a Fredholm kernel of type $L^{p,\eta}$ on $[0,\infty)$ if the function $(t,s)\mapsto e^{\eta(t-s)}k(t,s)$ is a Fredholm kernel of type $L^p$ on $[0,\infty)$. The following lemma shows the solvability of the deterministic Fredholm integral equation~\eqref{LQ: DFIE_abstract}. Recall that the constant $\rho_\alpha$ is defined as the positive number satisfying $|b|(1+e^{-\rho_\alpha\delta})\rho^{-\alpha}_\alpha=1$ if $b\neq0$, and $\rho_\alpha:=0$ if $b=0$.

%% Lemma

\begin{lemm}\label{LQ: lemm_resolvent}
Define $\tilde{\rho}_\alpha>\rho_\alpha$ as the number satisfying
\begin{equation}\label{LQ: criterion}
	\bigl\{\bigl(c^2\gamma+b^2e^{-2\tilde{\rho}_\alpha\delta}\bigr)(2\tilde{\rho}_\alpha)^{-\alpha}+|b|(1+e^{-\tilde{\rho}_\alpha\delta})\bigr\}\tilde{\rho}^{-\alpha}_\alpha=\frac{1}{2}.
\end{equation}
Then for any $\mu>\tilde{\rho}_\alpha$ and $\lambda\geq2\mu$, the following assertions hold:
\begin{itemize}
\item[(i)]
$k_\lambda$ is a Fredholm kernel of type $L^{2,-\mu}$ on $[0,\infty)$.
\item[(ii)]
There exists a unique \emph{Fredholm resolvent} $r_\lambda$ of type $L^{2,-\mu}$ of $k_\lambda$ on $[0,\infty)$, that is, $r_\lambda$ is the unique Fredholm kernel of type $L^{2,-\mu}$ on $[0,\infty)$ satisfying the resolvent equation
\begin{equation*}
	r_\lambda(t,s)+\int^\infty_0k_\lambda(t,\theta)r_\lambda(\theta,s)\rd\theta=r_\lambda(t,s)+\int^\infty_0r_\lambda(t,\theta)k_\lambda(\theta,s)\rd\theta=k_\lambda(t,s)
\end{equation*}
for a.e.\ $t,s\geq0$.
\item[(iii)]
For any $a\in L^{2,-\mu}(0,\infty;\bR)$, the Fredholm integral equation \eqref{LQ: DFIE_abstract} admits a unique solution $x$ in $L^{2,-\mu}(0,\infty;\bR)$. Furthermore, this solution is given by the variation of constant formula
\begin{equation*}
	x(t)=a(t)-\int^\infty_0r_\lambda(t,s)a(s)\rd s
\end{equation*}
for a.e.\ $t\geq0$.
\end{itemize}
\end{lemm}

%% Proof

\begin{proof}
Let $\mu>\tilde{\rho}_\alpha$ and $\lambda\geq2\mu$, and define $\tilde{k}_{\mu,\lambda}(t,s):=e^{-\mu(t-s)}k_\lambda(t,s)$ for $t,s\geq0$. Then the conditions (i), (ii) and (iii) in this lemma are equivalent to the following conditions, respectively:
\begin{itemize}
\item[(i)']
$\tilde{k}_{\mu,\lambda}$ is a Fredholm kernel of type $L^2$ on $[0,\infty)$.
\item[(ii)']
There exists a unique Fredholm resolvent $\tilde{r}_{\mu,\lambda}$ of type $L^2$ of $\tilde{k}_{\mu,\lambda}$ on $[0,\infty)$.
\item[(iii)']
For any $\tilde{a}\in L^2(0,\infty;\bR)$, the Fredholm integral equation
\begin{equation*}
	\tilde{x}(t)+\int^\infty_0\tilde{k}_{\mu,\lambda}(t,s)\tilde{x}(s)\rd s=\tilde{a}(t),\ t\geq0,
\end{equation*}
admits a unique solution $\tilde{x}$ in $L^2(0,\infty;\bR)$. Furthermore, this solution is given by the variation of constant formula
\begin{equation*}
	\tilde{x}(t)=\tilde{a}(t)-\int^\infty_0\tilde{r}_{\mu,\lambda}(t,s)\tilde{a}(s)\rd s
\end{equation*}
for a.e.\ $t\geq0$.
\end{itemize}
Indeed, the equivalence (i)\,$\Leftrightarrow$\,(i)' is clear from the definition, and (ii)\,$\Leftrightarrow$\,(ii)'  and (iii)\,$\Leftrightarrow$\,(iii)' easily follow by considering the relations $\tilde{r}_{\mu,\lambda}(t,s)=e^{-\mu(t-s)}r_\lambda(t,s)$, $\tilde{a}(t)=e^{-\mu t}a(t)$ and $\tilde{x}(t)=e^{-\mu t}x(t)$. For more detailed treatments of integral equations in weighted spaces, see Chapter~4 in \cite{GrLoSt90}.

We show that $\res\tilde{k}_{\mu,\lambda}\res_{L^2(0,\infty;\bR)}<1$. If this is the case, the assertions (i)', (ii)' and (iii)' follow from Corollary~9.3.10, Lemma~9.3.3, and Theorem~9.3.6 in \cite{GrLoSt90}. Note that
\begin{equation*}
	\tilde{k}_{\mu,\lambda}(t,s)=
	\begin{dcases}
	e^{-\mu(t-s)}g_\lambda(t-s)\ &\text{for}\ (t,s)\in\Delta^\comp[0,\infty),\\
	e^{-(\lambda-\mu)(s-t)}g_\lambda(s-t)\ &\text{for}\ (t,s)\in\Delta[0,\infty).
	\end{dcases}
\end{equation*}
Noting $\lambda\geq2\mu$ and \eqref{LQ: g_lambda_estimate}, we see that $|\tilde{k}_{\mu,\lambda}(t,s)|\leq M_\mu(t-s)$ for any $(t,s)\in\{(t,s)\in[0,\infty)^2\,|\,t\neq s\}$, where $M_\mu\in L^1(\bR;\bR)$ is defined by
\begin{equation*}
	M_\mu(x):=\frac{c^2\gamma+b^2e^{-2\mu\delta}}{\Gamma(\alpha)}(2\mu)^{-\alpha}e^{-\mu|x|}|x|^{\alpha-1}_++\frac{|b|}{\Gamma(\alpha)}e^{-\mu|x|}(|x|-\delta)^{\alpha-1}_+,\ x\in\bR.
\end{equation*}
Here, we recall the definition $x^{\alpha-1}_+:=x^{\alpha-1}\1_{(0,\infty)}(x)$ for any $x\in\bR$. Noting Proposition~9.2.7 in \cite{GrLoSt90}, we have
\begin{align*}
	\res\tilde{k}_{\mu,\lambda}\res_{L^\infty(0,\infty;\bR)}&=\underset{t\geq0}{\mathrm{ess\,sup}}\int^\infty_0|\tilde{k}_{\mu,\lambda}(t,s)|\rd s\\
	&\leq\int^\infty_{-\infty}M_\mu(x)\rd x=2\bigl\{\bigl(c^2\gamma+b^2e^{-2\mu\delta}\bigr)(2\mu)^{-\alpha}+|b|e^{-\mu\delta}\bigr\}\mu^{-\alpha}\\
	&<1,
\end{align*}
where the last inequality follows since $\mu>\tilde{\rho}_\alpha$ and $\tilde{\rho}_\alpha$ satisfies \eqref{LQ: criterion}.
Similarly, we have
\begin{equation*}
	\res\tilde{k}_{\mu,\lambda}\res_{L^1(0,\infty;\bR)}=\underset{s\geq0}{\mathrm{ess\,sup}}\int^\infty_0|\tilde{k}_{\mu,\lambda}(t,s)|\rd t\leq\int^\infty_{-\infty}M_\mu(x)\rd x<1.
\end{equation*}
Thus, the Fredholm kernel $\tilde{k}_{\mu,\lambda}$ is of both type $L^1$ and $L^\infty$ on $[0,\infty)$. By the Riesz interpolation theorem for Fredholm kernels (see Theorem~9.2.6 in \cite{GrLoSt90}), we see that $\tilde{k}_{\mu,\lambda}$ is of type $L^2$ on $[0,\infty)$ and satisfies
\begin{equation*}
	\res\tilde{k}_{\mu,\lambda}\res_{L^2(0,\infty;\bR)}\leq\res\tilde{k}_{\mu,\lambda}\res^{1/2}_{L^1(0,\infty;\bR)}\,\res\tilde{k}_{\mu,\lambda}\res^{1/2}_{L^\infty(0,\infty;\bR)}<1.
\end{equation*}
Therefore, we obtain the assertions.
\end{proof}

Now we are ready to state the main result of this subsection.

%% Theorem

\begin{theo}\label{LQ: theo_characterization}
Let $\tilde{\rho}_\alpha>\rho_\alpha $ be the constant specified by \eqref{LQ: criterion}, and fix $\lambda>2\tilde{\rho}_\alpha$. Then there exists a unique optimal control $\hat{u}(\cdot)$ of the LQ regulator problem \eqref{LQ: state}--\eqref{LQ: cost} over all control processes in $\cU_{-\lambda/2}$. Furthermore, the optimal control $\hat{u}(\cdot)$ belongs to the smaller control space $\cU_{-\mu}$ for any $\mu\in(\tilde{\rho}_\alpha,\lambda/2]$, and it is given by the following Gaussian state-feedback representation formula:
\begin{align*}
	\hat{u}(t)=&-\frac{b}{c}\hat{X}(t-\delta)+\frac{\bigl(c^2\gamma+b^2e^{-\lambda\delta}\bigr)\lambda^{-\alpha}-b}{c}\Bigl(\int^\infty_0r_\lambda(t,s)\rd s-1\Bigr)x_0\\
	&\hspace{1cm}+\frac{\sigma}{c}\int^t_0\Bigl\{\int^\infty_0r_\lambda(t-\theta,s)g_\lambda(s)\rd s-g_\lambda(t-\theta)\Bigr\}\rd W(\theta),\ t\geq0,
\end{align*}
where $\hat{X}(\cdot):=X^{\hat{u}}(\cdot)\in L^{2,-\mu}_\bF(0,\infty;\bR)$ is the optimal state process corresponding to $\hat{u}(\cdot)$, $g_\lambda:[0,\infty)\to\bR$ is defined by \eqref{LQ: g_lambda}, $r_\lambda:[0,\infty)^2\to\bR$ is the Fredholm resolvent of type $L^{2,-\mu}$ of $k_\lambda$ on $[0,\infty)$, and $k_\lambda:[0,\infty)^2\to\bR$ is the Fredholm kernel of type $L^{2,-\mu}$ on $[0,\infty)$ defined by \eqref{LQ: k_lambda}.
\end{theo}

%% Proof

\begin{proof}
Note that each real constant can be seen as an element of $\bigcap_{\beta<0}L^{2,\beta}(0,\infty;\bR)$. Also, by \cref{LQ: rem_g_lambda}, the function $g_\lambda$ is in $\bigcap_{\beta<0}L^{2,\beta}(0,\infty;\bR)$. Fix an arbitrary constant $\mu\in(\tilde{\rho}_\alpha,\lambda/2]$. By \cref{LQ: lemm_resolvent}, the deterministic Fredholm integral equations \eqref{LQ: DFIE_phi} and \eqref{LQ: DFIE_psi} admit unique solutions $\hat{\varphi}$ and $\hat{\psi}$ in $L^{2,-\mu}(0,\infty;\bR)$ given by
\begin{equation*}
	\hat{\varphi}(t)=-K_\lambda x_0+\int^\infty_0r_\lambda(t,s)K_\lambda x_0\rd s=\frac{\bigl(c^2\gamma+b^2e^{-\lambda\delta}\bigr)\lambda^{-\alpha}-b}{c}\Bigl(\int^\infty_0r_\lambda(t,s)\rd s-1\Bigr)x_0,\ t\geq0,
\end{equation*}
and
\begin{equation*}
	\hat{\psi}(t)=-\frac{\sigma}{c}g_\lambda(t)+\int^\infty_0r_\lambda(t,s)\frac{\sigma}{c}g_\lambda(s)\rd s=\frac{\sigma}{c}\Bigl(\int^\infty_0r_\lambda(t,s)g_\lambda(s)\rd s-g_\lambda(t)\Bigr),\ t\geq0,
\end{equation*}
respectively. Then, \cref{LQ: lemm_FIE} implies that the stochastic Fredholm integral equation \eqref{LQ: SFIE} admits a unique solution $\hat{v}(\cdot)$ in $\cU_{-\mu}$, and this solution is a Gaussian process of the form
\begin{align*}
	\hat{v}(t)&=\hat{\varphi}(t)+\int^t_0\hat{\psi}(t-\theta)\rd W(\theta)\\
	&=\frac{\bigl(c^2\gamma+b^2e^{-\lambda\delta}\bigr)\lambda^{-\alpha}-b}{c}\Bigl(\int^\infty_0r_\lambda(t,s)\rd s-1\Bigr)x_0\\
	&\hspace{1cm}+\frac{\sigma}{c}\int^t_0\Bigl\{\int^\infty_0r_\lambda(t-\theta,s)g_\lambda(s)\rd s-g_\lambda(t-\theta)\Bigr\}\rd W(\theta),\ t\geq0.
\end{align*}
Therefore, by \cref{LQ: lemm_OC2}, the control process $\hat{u}(\cdot):=\cT^{-1}[\hat{v}](\cdot)\in\cU_{-\mu}$ is optimal on the control space $\cU_{-\mu}$, and thus it is optimal on the larger control space $\cU_{-\lambda/2}$ too (see \cref{SVIE: rem_large}). Furthermore, by the definition of the bijective map $\cT$, together with the relation $\cX^{\hat{v}}(\cdot)=X^{\hat{u}}(\cdot)=:\hat{X}(\cdot)$, we obtain the following Gaussian state-feedback representation formula:
\begin{align*}
	\hat{u}(t)&=\cT^{-1}[\hat{v}](t)=-\frac{b}{c}\cX^{\hat{v}}(t-\delta)+\hat{v}(t)\\
	&=-\frac{b}{c}\hat{X}(t-\delta)+\frac{\bigl(c^2\gamma+b^2e^{-\lambda\delta}\bigr)\lambda^{-\alpha}-b}{c}\Bigl(\int^\infty_0r_\lambda(t,s)\rd s-1\Bigr)x_0\\
	&\hspace{1cm}+\frac{\sigma}{c}\int^t_0\Bigl\{\int^\infty_0r_\lambda(t-\theta,s)g_\lambda(s)\rd s-g_\lambda(t-\theta)\Bigr\}\rd W(\theta),\ t\geq0.
\end{align*}
The uniqueness of the optimal control follows from \cref{LQ: lemm_unique}, and thus we complete the proof.
\end{proof}

%% Remark

\begin{rem}\label{LQ: rem_Gaussian}
It is easy to see that the optimal state $\hat{X}(\cdot)$ is a Gaussian process. Indeed, under the notations in the above proof, we see that
\begin{align*}
	\hat{X}(t)&=\cX^{\hat{v}}(t)=x_0+\frac{c}{\Gamma(\alpha)}\int^t_0(t-s)^{\alpha-1}\hat{v}(s)\rd s+\frac{\sigma}{\Gamma(\alpha)}\int^t_0(t-s)^{\alpha-1}\rd W(s)\\
	&=x_0+\frac{c}{\Gamma(\alpha)}\int^t_0(t-s)^{\alpha-1}\Bigl\{\hat{\varphi}(s)+\int^s_0\hat{\psi}(s-\theta)\rd W(\theta)\Bigr\}\rd s+\frac{\sigma}{\Gamma(\alpha)}\int^t_0(t-s)^{\alpha-1}\rd W(s)\\
	&=x_0+\frac{c}{\Gamma(\alpha)}\int^t_0(t-s)^{\alpha-1}\hat{\varphi}(s)\rd s+\frac{1}{\Gamma(\alpha)}\int^t_0\Bigl\{c\int^t_\theta(t-s)^{\alpha-1}\hat{\psi}(s-\theta)\rd s+\sigma(t-\theta)^{\alpha-1}\Bigr\}\rd W(\theta)\\
	&=x_0+\frac{c}{\Gamma(\alpha)}\int^t_0(t-s)^{\alpha-1}\hat{\varphi}(s)\rd s+\frac{1}{\Gamma(\alpha)}\int^t_0\Bigl\{c\int^{t-\theta}_0(t-\theta-s)^{\alpha-1}\hat{\psi}(s)\rd s+\sigma(t-\theta)^{\alpha-1}\Bigr\}\rd W(\theta),
\end{align*}
where, in the third line, we used the stochastic Fubini's theorem. Note that $\hat{X}(\cdot)$ is of the form of a stochastic convolution.
\end{rem}

%% Remark

\begin{rem}\label{LQ: rem_summary}
The above method deriving the Gaussian state-feedback representation formula relies on the special structure of the infinite horizon LQ regulator problem \eqref{LQ: state}--\eqref{LQ: cost}. Let us summarize the above arguments and suggest further generalizations.
\begin{itemize}
\item[(i)]
The first step (\cref{LQ: lemm_OC1}) is based on the necessary and sufficient maximum principles (see \cref{SDVIE: theo_MP}). An idea here is to get rid of the adjoint variable $\hat{Y}(\cdot)$ so that the optimality condition is rewritten in terms of $\hat{u}(\cdot)$ and $\hat{X}(\cdot)$ only. This can be done since the diffusion term does not depend on the state or control, since the control set $U$ is the whole real line, and since the SDVIE~\eqref{LQ: state'} is of the convolution type.
\item[(ii)]
The second step (\cref{LQ: lemm_OC2}) is to consider the bijective transformation $\hat{v}(\cdot):=\cT[\hat{u}](\cdot)$ of $\hat{u}(\cdot)$ so that the state variable $\hat{X}(\cdot)$ in the above optimality condition vanishes. As a result, we get the stochastic Fredholm integral equation \eqref{LQ: SFIE} for $\hat{v}(\cdot)$, which characterizes the optimality of $\hat{u}(\cdot)=\cT^{-1}[\hat{v}](\cdot)$. In this step, the infinite horizon framework makes some computations simple. However, the finite horizon case can be treated by the same manner.
\item[(iii)]
The third step (\cref{LQ: lemm_FIE}) is to divide the stochastic Fredholm integral equation~\eqref{LQ: SFIE} into two deterministic Fredholm integral equations \eqref{LQ: DFIE_phi} and \eqref{LQ: DFIE_psi} which correspond to the ``expectation part'' and the ``martingale integrand part'' of $\hat{v}(\cdot)$, respectively. Here, the special parametrization of the martingale integrand of $\hat{v}(\cdot)$ (see \cref{LQ: lemm_MR}) plays a crucial role. We can easily treat this parametrization since it is defined on the infinite horizon. In the finite horizon case, we have to be careful for the treatment of the domain of the two time-parameters. Also, we note that one of the main reason why deterministic equations appear is that all the coefficients of the problem are deterministic.
\item[(iv)]
The last step (\cref{LQ: lemm_resolvent}) is to solve the linear deterministic Fredholm integral equations \eqref{LQ: DFIE_phi} and \eqref{LQ: DFIE_psi} by means of the Fredholm resolvent. This step is based on the well-established general theory in the textbook \cite{GrLoSt90}.
\end{itemize}
We note that, unlike the related papers \cite{BoCoMa12,CoMa11,Ma19,AbMiPh19}, the complete monotonicity of the fractional kernel $\tau^{\alpha-1}$ seems to be not necessary in the above arguments. Also, the multi-dimensional case can be treated by the same manner. Therefore, we conjecture that if the diffusion term does not depend on the state or control, multi-dimensional linear SDVIEs with (not necessarily completely monotone) convolution kernels can be treated by the same manner as above. We hope to report some relevant results for more general cases in the near future.
\end{rem}

\section*{Acknowledgments}
%The authors would like to thank the editor and the referees for their constructive comments and suggestions.
The author was supported by JSPS KAKENHI Grant Number 21J00460 and also partly by JSPS KAKENHI Grant Number 22K13958.

%%%%%%%%%%%%%%%%%%%%%%%%%%%%
%%%%%% References
%%%%%%%%%%%%%%%%%%%%%%%%%%%%

\end{document}